\documentclass{gtpart}
\usepackage{pinlabel}
\usepackage{graphicx}
\title[Virtual braids and permutations]{Virtual braids and permutations}
\author[P Bellingeri]{Paolo Bellingeri}
\givenname{Paolo}
\surname{Bellingeri}
\address{LMNO, UMR 6139, CNRS, Universit\'e de Caen-Normandie, 14000 Caen, France}
\email{paolo.bellingeri@unicaen.fr}	
\urladdr{}
\author[L Paris]{Luis Paris}
\givenname{Luis}
\surname{Paris}
\address{IMB, UMR 5584, CNRS, Univ. Bourgogne Franche-Comté, 21000 Dijon, France}
\email{lparis@u-bourgogne.fr}
\urladdr{}

\volumenumber{}
\issuenumber{}
\publicationyear{}
\papernumber{}
\startpage{}
\endpage{}
\doi{}
\MR{}
\Zbl{}
\received{}
\revised{}
\accepted{}
\published{}
\publishedonline{}
\proposed{}
\seconded{}
\corresponding{}
\editor{}
\version{}
\newtheorem{thm}{Theorem}[section]
\newtheorem{lem}[thm]{Lemma}
\newtheorem{prop}[thm]{Proposition}
\newtheorem{corl}[thm]{Corollary}

\theoremstyle{definition}

\newtheorem*{expl}{Example}

\numberwithin{equation}{section}

\makeatletter
\renewcommand{\thefigure}{\ifnum \c@section>\z@ \thesection.\fi
 \@arabic\c@figure}
\@addtoreset{figure}{section}
\makeatother

\begin{document}

\def\SSS{\mathfrak S} \def\VB{{\rm VB}} \def\Out{{\rm Out}}
\def\Z{\mathbb Z} \def\N{\mathbb N} \def\VP{{\rm VP}}
\def\KB{{\rm KB}} \def\Aut{{\rm Aut}} \def\id{{\rm id}}
\def\SS{\mathcal S} \def\XX{\mathcal X} \def\YY{\mathcal Y}
\def\UU{\mathcal U} \def\VV{\mathcal V} \def\WW{\mathcal W}


\begin{abstract}
Let $\VB_n$ be the virtual braid group on $n$ strands and let $\SSS_n$ be the symmetric group on $n$ letters.
Let $n,m \in \N$ such that $n \ge 5$, $m \ge 2$ and $n \ge m$.
We determine all possible homomorphisms from $\VB_n$ to $\SSS_m$, from  $\SSS_n$ to $\VB_m$ and  from  $\VB_n$ to $\VB_m$.
As corollaries we get that $\Out(\VB_n)$ is isomorphic to $\Z/2\Z \times \Z/2\Z$ and that $\VB_n$ is both, Hopfian and co-Hofpian.
\end{abstract}

\maketitle


\section{Introduction}\label{sec1}

The study of the homomorphisms from the braid group $B_n$ on $n$ strands to the symmetric group $\SSS_n$ goes back to Artin \cite{Artin1} himself. 
He was right thinking that this would be an important step toward the determination of the automorphism group of $B_n$.
As pointed out by Lin \cite{Lin1,Lin2}, all the homomorphisms from $B_n$ to $\SSS_m$ for $n \ge m$ are easily deduced from the ideas of Artin \cite{Artin1}. 
The automorphism group of $B_n$ was then determined by Dyer--Grossman \cite{DyeGro1} using in particular Artin's results in \cite{Artin1}.
The homomorphisms from $B_n$ to $B_m$ were determined by Lin \cite{Lin1,Lin2} for $n >m$ and by Castel \cite{Caste1} for $n=m$ and $n \ge 6$.

Virtual braids were introduced by Kauffman \cite{Kauff1} together with virtual knots and links. 
They have interpretations in terms of diagrams (see Kauffman \cite{Kauff1}, Kamada \cite{Kamad1} and Vershinin \cite{Versh1}) and also in terms of braids in thickened surfaces (see Cisneros de la Cruz \cite{Cisne1}) and now there is a quite extensive literature on them.
Despite their interest in low-dimensional topology, the virtual braid groups are poorly understood from a combinatorial point of view.
Among the known results for these groups there are solutions to the word problem in Bellingeri--Cisneros de la Cruz--Paris \cite{BeCiPa1} and in Chterental \cite{Chter1} and the calculation of some terms of its lower central series in Bardakov--Bellingeri \cite{BarBel1}, and some other results but not many more. 
For example, it is not known whether these groups have a solution to the conjugacy problem or whether they are linear.

In the present paper we prove results on virtual braid groups in the same style as those previously mentioned for the braid groups.
More precisely, we take $n, m \in \N$ such that $n \ge 5$, $m \ge 2$ and $n \ge m$ and 
\begin{itemize}
\item[(1)]
we determine all the homomorphisms from $\VB_n$ to $\SSS_m$,
\item[(2)]
we determine all the homomorphisms from $\SSS_n$ to $\VB_m$,
\item[(3)]
we determine all the homomorphisms from $\VB_n$ to $\VB_m$.
\end{itemize}
From these classifications it will follow that $\Out (\VB_n)$ is isomorphic to $\Z/2\Z \times \Z/2\Z$, and that $\VB_n$ is both, Hopfian and co-Hopfian.

Our study is totally independent from the works of Dyer--Grossman \cite{DyeGro1}, Lin \cite{Lin1, Lin2} and  Castel \cite{Caste1} cited above and all similar works on braid groups. 
Our viewpoint/strategy is also new for virtual braid groups, although some aspects are already present in  Bellingeri--Cisneros de la Cruz--Paris \cite{BeCiPa1}.
Our idea/main-contribution consists on first observing that the virtual braid group $\VB_n$ decomposes as a semi-direct product $\VB_n = \KB_n \rtimes \SSS_n$ of an Artin group $\KB_n$ by the symmetric group $\SSS_n$, and secondly and mainly on making a deep study of the action of $\SSS_n$ on $\KB_n$. 
From this study we deduce that there exists a unique embedding of the symmetric group $\SSS_n$ into the virtual braid group $\VB_n$ up to conjugation (see Lemma \ref{lem5_1}).
From there the classifications are deduced with some extra work.

The group $\VB_2$ is isomorphic to $\Z * \Z/2\Z$.
On the other hand, fairly precise studies of the combinatorial structure of $\VB_3$ can be found in Bardakov--Bellingeri \cite{BarBel1}, Bardakov--Mikhailov--Vershinin--Wu \cite{BMVW1} and Bellingeri--Cisneros de la Cruz--Paris \cite{BeCiPa1}.
From this one can probably determine (with some difficulties) all the homomorphisms from $\VB_n$ to $\SSS_m$,  from $\SSS_n$ to $\VB_m$, and from $\VB_n$ to $\VB_m$, for $n \ge m$ and $n = 2,3$. 
Our guess is that the case $n=4$ will be much more tricky.

Our paper is organized as follows. 
In Section \ref{sec2} we give the definitions of the homomorphisms involved in the paper, we state our three main theorems, and we prove their corollaries. 
Section \ref{sec3} is devoted to the study of the action of the symmetric group $\SSS_n$ on the above mentioned group $\KB_n$.
The results of this section are technical but they provide key informations on the structure of $\VB_n$ that are essential in the proofs of our main theorems. 
We also think that they are interesting by themselves and may be used for other purposes. 
Our main theorems are proved in Section \ref{sec4}, Section \ref{sec5} and Section \ref{sec6}, respectively. 


\section{Definitions and statements}\label{sec2}

In the paper we are interested in the algebraic and combinatorial aspects of virtual braid groups, hence we will adopt their standard definition in terms of generators and relations.
So, the \emph{virtual braid group} $\VB_n$ on $n$ strands is the group defined by the presentation with generators $\sigma_1, \dots, \sigma_{n-1}, \tau_1, \dots, \tau_{n-1}$ and relations 
\begin{gather*}
\tau_i^2 = 1 \text{ for } 1 \le i \le n-1\,,\\
\sigma_i \sigma_j = \sigma_j \sigma_i\,, \ 
\tau_i \tau_j = \tau_j \tau_i\,,\
\tau_i \sigma_j = \sigma_j \tau_i \text{ for } |i-j| \ge 2\,,\\
\sigma_i \sigma_j \sigma_i = \sigma_j \sigma_i \sigma_j\,,\
\tau_i \tau_j \tau_i = \tau_j \tau_i \tau_j\,,\
\tau_i \tau_j \sigma_i = \sigma_j \tau_i \tau_j \text{ for }|i-j| = 1\,.
\end{gather*}

For each $i \in \{1, \dots, n-1\}$ we set $s_i = (i,i+1)$.
It is well-known that the symmetric group $\SSS_n$ on $n$ letters has a presentation with generators $s_1, \dots, s_{n-1}$ and relations 
\[
s_i^2 = 1 \text{ for } 1 \le i \le n-1\,,\
s_i s_j = s_j s_i \text{ for } |i-j| \ge 2\,,\
s_i s_j s_i = s_j s_i s_j \text{ for } |i-j| = 1\,.
\]

Now, we define the homomorphisms that are involved in the paper. 
Let $G,H$ be two groups. 
We say that a homomorphism $\psi : G \to H$ is \emph{Abelian} if the image of $\psi$ is an Abelian subgroup of $H$.
Note that the abelianization of $\SSS_n$ is isomorphic to $\Z /2\Z$, hence the image of any Abelian homomorphism $\varphi : \SSS_n \to H$ is either trivial or a cyclic group of order $2$.  
On the other hand, the abelianization of $\VB_n$ is isomorphic to $\Z \times \Z/2\Z$, where the copy of $\Z$ is generated by the class of $\sigma_1$ and the copy of $\Z /2\Z$ is generated by the class of $\tau_1$.
Thus, the image of any Abelian homomorphism $\psi : \VB_n \to H$ is a quotient of $\Z \times \Z / 2\Z$.

From the presentations of $\VB_n$ and $\SSS_n$ given above we see that there are epimorphisms $\pi_P : \VB_n \to \SSS_n$ and $\pi_K : \VB_n \to \SSS_n$ defined by $\pi_P (\sigma_i) = \pi_P (\tau_j) = s_i$ for all $1 \le i \le n-1$ and by $\pi_K (\sigma_i) = 1$ and $\pi_K (\tau_j) = s_i$ for all $1 \le i \le n-1$, respectively. 
The kernel of $\pi_P$ is called the \emph{virtual pure braid group} and is denoted by $\VP_n$.
A presentation of this group can be found in Bardakov \cite{Barda1}.
It is isomorphic to the group of the Yang--Baxter equation studied in Bartholdi--Enriquez--Etingof--Rains \cite{BEER1}.
The kernel of $\pi_K$ does not have any particular name. 
It is denoted by $\KB_n$.
It is an Artin group, hence we can use tools from the theory of Artin groups to study it and get results on $\VB_n$ itself (see for example Godelle--Paris \cite{GodPar1} or Bellingeri--Cisneros de la Cruz--Paris \cite{BeCiPa1}).
This group will play a prominent role in our study.
It will be described and studied in Section \ref{sec3}.

Using again the presentations of $\SSS_n$ and $\VB_n$ we see that there is a homomorphism $\iota : \SSS_n \to \VB_n$ that sends $s_i$ to $\tau_i$ for all $1 \le i \le n-1$.
Observe that $\iota$ is a section of both, $\pi_P$ and $\pi_K$, hence $\iota$ is injective and we have the decompositions $\VB_n = \VP_n \rtimes \SSS_n$ and $\VB_n = \KB_n \rtimes \SSS_n$.

The two main automorphisms of $\VB_n$ that are involved in the paper are the automorphisms $\zeta_1, \zeta_2: \VB_n \to \VB_n$ defined by $\zeta_1 (\sigma_i) = \tau_i \sigma_i \tau_i$ and $\zeta_1 (\tau_i) = \tau_i$ for all $1 \le i \le n-1$, and by $\zeta_2 (\sigma_i) = \sigma_i^{-1}$ and $\zeta_2 (\tau_i) = \tau_i$ for all $1 \le i \le n-1$, respectively. 
It is easily checked that $\zeta_1$ and $\zeta_2$ are of order two and commute, hence they generate a subgroup of $\Aut (\VB_n)$ isomorphic to $\Z / 2 \Z \times \Z / 2 \Z$.

The last homomorphism (automorphism) concerned by the paper appears only for $n=6$.
This is the automorphism $\nu_6 : \SSS_6 \to \SSS_6$ defined by 
\begin{gather*}
\nu_6 (s_1) = (1,2)(3,4)(5,6)\,,\
\nu_6 (s_2) = (2,3)(1,5)(4,6)\,,\
\nu_6 (s_3) = (1,3)(2,4)(5,6)\,,\\
\nu_6 (s_4) = (1,2)(3,5)(4,6)\,,\
\nu_6 (s_5) = (2,3)(1,4)(5,6)\,.
\end{gather*}
It is well-known that $\Out (\SSS_n)$ is trivial for $n \neq 6$ and that $\Out (\SSS_6)$ is a cyclic group of order $2$ generated by the class of $\nu_6$.
Note also that $\nu_6^2$ is the conjugation by $w_0 = (1,6,2,5,3)$.

We give a last definition before passing to the statements.
Let $G, H$ be two groups. 
For each $\alpha \in H$ we denote by $c_\alpha : H \to H$, $\beta \mapsto \alpha \beta \alpha^{-1}$, the conjugation by $\alpha$.
We say that two homomorphisms $\psi_1, \psi_2 : G \to H$ are \emph{conjugate} and we write $\psi_1 \sim_c \psi_2$ if there exists $\alpha \in H$ such that $\psi_2 = c_\alpha \circ \psi_1$.

In Section \ref{sec4} we determine the homomorphisms from $\VB_n$ to $\SSS_m$ up to conjugation. 
More precisely we prove the following.

\begin{thm}\label{thm2_1}
Let $n,m \in \N$ such that $n \ge 5$, $m \ge 2$ and $n \ge m$, and let $\psi : \VB_n \to \SSS_m$ be a homomorphism. 
Then, up to conjugation, one of the following possibilities holds.
\begin{itemize}
\item[(1)]
$\psi$ is Abelian,
\item[(2)]
$n=m$ and $\psi \in \{ \pi_K, \pi_P \}$,
\item[(3)]
$n=m=6$ and $\psi \in \{ \nu_6 \circ \pi_K, \nu_6 \circ \pi_P \}$.
\end{itemize}
\end{thm}

In Section \ref{sec5} we determine the homomorphisms from $\SSS_n$ to $\VB_m$ up to conjugation. 
More precisely we prove the following.

\begin{thm}\label{thm2_2} 
Let $n,m \in \N$ such that $n \ge 5$, $m \ge 2$ and $n \ge m$, and let $\varphi : \SSS_n \to \VB_m$ be a homomorphism. 
Then, up to conjugation, one of the following possibilities holds. 
\begin{itemize}
\item[(1)]
$\varphi$ is Abelian, 
\item[(2)]
$n=m$ and $\varphi = \iota$,
\item[(3)]
$n=m=6$ and $\varphi = \iota \circ \nu_6$.
\end{itemize}
\end{thm}

Finally, in Section \ref{sec6} we determine the homomorphisms from $\VB_n$ to $\VB_m$ up to conjugation. 
More precisely we prove the following.

\begin{thm}\label{thm2_3}
Let $n,m \in \N$ such that $n \ge 5$, $m \ge 2$ and $n \ge m$, and let $\psi : \VB_n \to \VB_m$ be a homomorphism. 
Then, up to conjugation, one of the following possibilities holds. 
\begin{itemize}
\item[(1)]
$\psi$ is Abelian, 
\item[(2)]
$n=m$ and $\psi \in \{ \iota \circ \pi_K, \iota \circ \pi_P \}$,
\item[(3)]
$n=m=6$ and $\psi \in \{ \iota \circ \nu_6 \circ \pi_K, \iota \circ \nu_6 \circ \pi_P\}$,
\item[(4)]
$n=m$ and $\psi \in \{ \id, \zeta_1, \zeta_2, \zeta_1 \circ \zeta_2 \} = \langle \zeta_1, \zeta_2 \rangle$.
\end{itemize}
\end{thm}

A group $G$ is called \emph{Hopfian} if every surjective homomorphism $\psi : G \to G$ is also injective.
On the other hand, it is called \emph{co-Hopfian} if every injective homomorphism $\psi : G \to G$ is also surjective. 
It is known that the braid group $B_n$ is Hopfian but not co-Hopfian. 
However, it is quasi-co-Hopfian by Bell--Margalit \cite{BelMar1}.
The property of hopfianity for the braid group $B_n$ follows from the fact that it can be embedded in the automorphism group of the free group $F_n$ (see Artin \cite{Artin2}).
We cannot apply such an argument for the virtual braid group $\VB_n$ since we do not know if it can be embedded into the automorphism group of a finitely generated free group.
A first consequence of Theorem \ref{thm2_3} is the following.

\begin{corl}\label{corl2_4}
Let $n \in \N$, $n \ge 5$.
Then $\VB_n$ is Hopfian and co-Hopfian.
\end{corl}

\begin{proof}
We see in Theorem \ref{thm2_3} that, up to conjugation, the only surjective homomorphisms from $\VB_n$ to $\VB_n$ are the elements of $\{ \id, \zeta_1, \zeta_2, \zeta_1 \circ \zeta_2 \} = \langle \zeta_1, \zeta_2 \rangle$, and they are all automorphisms, hence $\VB_n$ is Hopfian. 
We show in the same way that $\VB_n$ is co-Hopfian. 
\end{proof}

Recall that, by Dyer--Grossman \cite{DyeGro1}, the group $\Out (B_n)$ is isomorphic to $\Z / 2 \Z$.
Here we show that $\Out (\VB_n)$ is a little bigger, that is:

\begin{corl}\label{corl2_5}
Let $n \in \N$, $n \ge 5$.
Then $\Out (\VB_n)$ is isomorphic to $\Z / 2 \Z \times \Z / 2 \Z$, and is generated by the classes of $\zeta_1$ and $\zeta_2$.
\end{corl}

The proof of Corollary \ref{corl2_5} relies on the following lemma whose proof is given in Section \ref{sec3}.

\begin{lem}\label{lem2_6}
Let $n \in \N$, $n \ge 5$.
Then $\zeta_1$ is not an inner automorphism of $\VB_n$.
\end{lem}

\begin{proof}[Proof of Corollary \ref{corl2_5}]
It follows from Theorem \ref{thm2_3} that $\Out (\VB_n)$ is generated by the classes of $\zeta_1$ and $\zeta_2$.
We also know that these two automorphisms are of order two and commute. 
So, it suffices to show that none of the elements $\zeta_1, \zeta_2, \zeta_1 \circ \zeta_2$ is an inner automorphism. 
It is easily seen from its presentation that the abelianization of $\VB_n$ is isomorphic to $\Z \times \Z / 2 \Z$, where the copy of $\Z$ is generated by the class of $\sigma_1$ and the copy of $\Z / 2 \Z$ is generated by the class of $\tau_1$.
Since $\zeta_2 (\sigma_1) = \sigma_1^{-1}$, $\zeta_2$ acts non-trivially on the abelianization of $\VB_n$, hence $\zeta_2$ is not an inner automorphism. 
The transformation $\zeta_1 \circ \zeta_2$ is not an inner automorphism for the same reason, and $\zeta_1$ is not an inner automorphism by Lemma \ref{lem2_6}.
\end{proof}

In order to determine $\Out (B_n)$, Dyer and Grossman \cite{DyeGro1} use also another result of Artin \cite{Artin1} which says that the pure braid group on $n$ strands is a characteristic subgroup of $B_n$.
From Theorem \ref{thm2_3} it immediately follows that the equivalent statement for virtual braid groups holds. 
More precisely, we have the following.

\begin{corl}\label{corl2_7}
Let $n \in \N$, $n \ge 5$.
Then the groups $\VP_n$ and $\KB_n$ are both characteristic subgroups of $\VB_n$.
\end{corl}


\section{Preliminaries}\label{sec3}

Let $n \ge 3$.
Recall that we have an epimorphism $\pi_K : \VB_n \to \SSS_n$ which sends $\sigma_i$ to $1$ and $\tau_i$ to $s_i$ for all $1 \le i \le n-1$.
Recall also that $\KB_n$ denotes the kernel of $\pi_K$ and that we have the decomposition $\VB_n = \KB_n
\rtimes \SSS_n$.
The aim of the present section is to prove three technical results on the action of $\SSS_n$ on $\KB_n$ (Lemma \ref{lem3_7}, Lemma \ref{lem3_9} and Lemma \ref{lem3_11}). 
These three lemmas are key points in the proofs of Theorem \ref{thm2_2} and Theorem \ref{thm2_3}.
We think also that they are interesting by themselves and may be use in the future for other purposes.
The two main tools in the proofs of these lemmas are the Artin groups and the amalgamated products.

The following is proved in Rabenda's master thesis at the Universit\'e de Bourgogne in Dijon in 2003.
This thesis is actually unavailable but the proof of the proposition can also be found in Bardakov--Bellingeri \cite{BarBel1}.

\begin{prop}\label{prop3_1}
For $1 \le i < j \le n$ we set 
\[ 
\begin{array}{rcl}
\delta_{i,j} &=& \tau_i \tau_{i+1} \cdots \tau_{j-2} \sigma_{j-1} \tau_{j-2} \cdots \tau_{i+1} \tau_i\,,\\
\noalign{\smallskip}
\delta_{j,i} &=& \tau_i \tau_{i+1} \cdots \tau_{j-2}\tau_{j-1} \sigma_{j-1} \tau_{j-1} \tau_{j-2} \cdots \tau_{i+1} \tau_i\,.\\
\end{array}\]
Then $\KB_n$ has a presentation with generating set
\[
\SS = \{ \delta_{i,j} \mid 1 \le i \neq j \le n\}\,,
\]
and relations 
\[\begin{array}{cl}
\delta_{i,j} \delta_{k,\ell} = \delta_{k,\ell} \delta_{i,j} &\quad \text{for } i,j,k,\ell \text{ pairwise distinct}\,,\\
\noalign{\smallskip}
\delta_{i,j} \delta_{j,k} \delta_{i,j} = \delta_{j,k} \delta_{i,j} \delta_{j,k} &\quad \text{for } i,j,k \text{ pairwise distinct}\,.
\end{array}
\]
Moreover, the action of each $w \in \SSS_n$ on a generator $\delta_{i,j}$ is by permutation of the indices, that is, $w (\delta_{i,j}) =\delta_{w(i),w(j)}$.
\end{prop}

Let $\SS$ be a finite set.  
A \emph{Coxeter matrix} over $\SS$ is a square matrix $M=(m_{s,t})_{s,t \in \SS}$ indexed by the elements of $\SS$ such that $m_{s,s}=1$ for all $s \in \SS$ and $m_{s,t} = m_{t,s} \in \{2,3,4, \dots\} \cup \{ \infty\}$ for all $s,t \in \SS$, $s \neq t$.
For $s,t \in \SS$, $s \neq t$, and $m \ge 2$, we denote by $\Pi(s,t,m)$ the word $sts \cdots$ of length $m$.
In other words, we have $\Pi (s,t,m) = (st)^{\frac{m}{2}}$ if $m$ is even and $\Pi(s,t,m) = (st)^{\frac{m-1}{2}}s$ if $m$ is odd.
The \emph{Artin group} associated with the Coxeter matrix  $M = (m_{s,t})_{s,t \in \SS}$ is the group $A=A_M$ defined by the presentation
\[
A = \langle \SS \mid \Pi (s,t, m_{s,t}) = \Pi (t,s,m_{s,t}) \text{ for } s,t \in \SS,\ s \neq t,\ m_{s,t} \neq \infty \rangle\,.
\]

\begin{expl}
Let $\SS = \{ \delta_{i,j} \mid 1 \le i \neq j \le n\}$.
Define a Coxeter matrix $M=(m_{s,t})_{s,t \in \SS}$ as follows.
Set  $m_{s,s}=1$ for all $s \in \SS$.
Let $s,t \in \SS$, $s \neq t$.
If $s=\delta_{i,j}$ and $t=\delta_{k,\ell}$, where $i,j,k,\ell$ are pairwise distinct, then set $m_{s,t}=2$.
If $s=\delta_{i,j}$ and $t=\delta_{j,k}$, where $i,j,k$ are pairwise distinct, then set $m_{s,t}=m_{t,s}=3$.
Set $m_{s,t} = \infty$ for all other cases.
Then, by Proposition \ref{prop3_1}, $\KB_n$ is the Artin group associated with $M = (m_{s,t})_{s,t \in \SS}$.
\end{expl}

If $\XX$ is a subset of $\SS$, then we set $M[\XX] = (m_{s,t})_{s,t \in \XX}$ and we denote by $A[\XX]$ the subgroup of $A$ generated by $\XX$.

\begin{thm}[Van der Lek \cite{Lek1}]\label{thm3_2}
Let  $A$ be the Artin group associated with a Coxeter matrix $M=(m_{s,t})_{s,t \in \SS}$, and let $\XX$ be a subset of $\SS$.
Then $A[\XX]$ is the Artin group associated with $M[\XX]$.
Moreover, if $\XX$ and $\YY$ are two subsets of $\SS$, then $A[\XX] \cap A[\YY] = A [\XX \cap \YY]$.
\end{thm}

In our example, for $\XX \subset \SS$ we denote by $\KB_n[\XX]$ the subgroup of $\KB_n$ generated by $\XX$. 
So, by Theorem \ref{thm3_2}, $\KB_n[\XX]$ is still an Artin group and it has a presentation with generating set $\XX$ and only two types of relations:
\begin{itemize}
\item
$st=ts$ if $m_{s,t}=2$ in $M[\XX]$,
\item
$sts=tst$ if $m_{s,t}=3$ in $M[\XX]$.
\end{itemize}

The following result can be easily proved from Theorem \ref{thm3_2} using presentations, but it is important to highlight it since it will be often used throughout the paper.

\begin{lem}\label{lem3_3}
Let $A$ be an Artin group associated with a Coxeter matrix $M=(m_{s,t})_{s,t \in \SS}$.
Let $\XX$ and $\YY$ be two subsets of $\SS$ such that $\XX \cup \YY = \SS$ and $m_{s,t} = \infty$ for all $s \in \XX \setminus (\XX \cap \YY)$ and $t \in \YY \setminus (\XX \cap \YY)$. 
Then $A = A[\XX]*_{A[\XX \cap \YY]} A[\YY]$.
\end{lem}

Let $G$ be a group and let $H$ be a subgroup of $G$.
A \emph{transversal} of $H$ in $G$ is a subset $T$ of $G$ such that for each $\alpha \in G$ there exists a unique $\theta \in T$ such that $\alpha H = \theta H$. 
For convenience we will always suppose that a transversal $T$ contains $1$ and we set $T^* = T \setminus \{1\}$.
The following is classical in the theory and is proved in Serre \cite[Section 1.1, Theorem 1]{Serre1}.

\begin{thm}[Serre \cite{Serre1}]\label{thm3_4}
Let $G_1, \dots, G_p, H$ be a collection of groups.
We suppose that $H$ is a subgroup of $G_j$ for all $j \in \{ 1, \dots, p \}$ and we consider the amalgamated product $G = G_1*_H G_2 *_H \cdots *_H G_p$.
For each $j \in \{1, \dots, p\}$ we choose a transversal $T_j$ of $H$ in $G_j$.
Then each element $\alpha \in G$ can be written in a unique way in the form $\alpha = \theta_1 \theta_2 \cdots \theta_\ell \beta$ such that:
\begin{itemize}
\item[(a)]
$\beta \in H$ and, for each $i \in \{1, \dots, \ell\}$, there exists $j = j(i) \in \{1, \dots, p \}$ such that $\theta_i \in T_j^* = T_j \setminus \{ 1 \}$,
\item[(b)]
$j(i) \neq j(i+1)$ for all  $i \in \{1, \dots, \ell-1\}$.
\end{itemize}
In particular, we have $\alpha \in H$ if and only if  $\ell = 0$ and $\alpha = \beta$.
\end{thm}

The expression of $\alpha$ given in the above theorem is called the \emph{normal form} of $\alpha$.
It depends on the amalgamated product and on the choice of the transversal of $H$ in $G_j$ for all $j$.

Consider the notation introduced in Theorem \ref{thm3_4}.
So, $G = G_1 *_H \cdots *_H G_p$ and $T_j$ is a transversal of $H$ in $G_j$ for all $j \in \{1, \dots, p\}$.
Let $\alpha \in G$.
We suppose that $\alpha$ is written in the form $\alpha = \alpha_1 \cdots \alpha_\ell$ such that $\ell \ge 1$, 
\begin{itemize}
\item[(a)]
for each $i \in \{1, \dots, \ell\}$ there exists $j=j(i) \in \{1, \dots, p\}$ such that $\alpha_i \in G_j \setminus H$,
\item[(b)]
$j(i) \neq j(i+1)$ for all  $i \in \{1, \dots, \ell-1\}$.
\end{itemize}
We define $\beta_i \in H$ and $\theta_i \in T_{j(i)}^*$ for $i \in \{1, \dots, \ell\}$ by induction on $i$ as follows.
First, $\theta_1$ is the element of $T_{j(1)}^*$ such that $\alpha_1 H = \theta_1 H$ and $\beta_1$ is the element of $H$ such that $\alpha_1 = \theta_1 \beta_1$.
We suppose that $i \ge 2$ and that $\beta_{i-1} \in H$ is defined.
Then $\theta_i$ is the element of $T_{j(i)}^*$ such that $\beta_{i-1} \alpha_i H = \theta_i H$ and $\beta_i$ is the element of $H$ such that $\beta_{i-1} \alpha_i = \theta_i \beta_i$.
The following result is quite obvious and its proof is left to the reader.

\begin{prop}\label{prop3_5}
Under the above notations and hypothesis $\alpha = \theta_1 \cdots \theta_\ell \beta_\ell$ is the normal form of  $\alpha$.
In particular, $\alpha \not\in H$ (since $\ell \ge 1$).
\end{prop}

We turn now to the proofs of our technical lemmas.

\begin{lem}\label{lem3_6}
Let $G_1, G_2, H$ be three groups.
We suppose that $H$ is a common subgroup of $G_1$ and $G_2$ and we set $G = G_1 *_H G_2$.
Let $\tau : G \to G$ be an automorphism of order $2$ such that $\tau (G_1) = G_2$ and $\tau(G_2) = G_1$.
Let  $G^\tau = \{ \alpha \in G \mid \tau(\alpha) = \alpha \}$.
Then $G^\tau \subset H$.
\end{lem}

\begin{proof}
Since $\tau (G_1) = G_2$, $\tau (G_2) = G_1$ and $H = G_1 \cap G_2$, we have $\tau(H) = H$.
Let $T_1$ be a transversal of $H$ in $G_1$.
Set $T_2 = \tau (T_1)$.
Then $T_2$ is a transversal of $H$ in $G_2$. 
Let $\alpha \in G$ and let $\alpha = \theta_1 \theta_2 \cdots \theta_\ell \beta$ be the normal form of $\alpha$.
Notice that the normal form of $\tau(\alpha)$ is $\tau(\theta_1)\, \tau (\theta_2) \cdots \tau(\theta_\ell)\, \tau (\beta)$.
Suppose that $\ell \ge 1$.
Without loss of generality we can assume that $\theta_1 \in T_1^*$.
Then $\tau(\theta_1) \in T_2^*$, hence $\theta_1 \neq \tau (\theta_1)$, thus $\alpha \neq \tau (\alpha)$, and therefore $\alpha \not\in G^\tau$.
So, if $\alpha \in G^\tau$, then $\ell =0$ and $\alpha = \beta \in H$.
\end{proof}

Recall from Proposition \ref{prop3_1} that $\VB_n = \KB_n \rtimes \SSS_n$ and that the action of $\SSS_n$ on $\KB_n$ is defined by $w(\delta_{i,j}) = \delta_{w(i), w(j)}$ for all $w \in \SSS_n$ and all $\delta_{i,j} \in \SS$.
Recall also that for each generator $s_i$ of $\SSS_n$ and each $\alpha \in \KB_n$ we have $s_i(\alpha) =\tau_i \alpha \tau_i$. 
Throughout the paper we will use both interpretations, as action of $s_i$ and as conjugation by $\tau_i$.
For each $1 \le k \le n$ we set $\SS_k = \{ \delta_{i,j} \mid k \le i \neq j \le n \}$.

\begin{lem}\label{lem3_7}
Let $\XX$ be a subset of $\SS$ invariant under the action of $s_1$.
Then $\KB_n [\XX]^{s_1} = \KB_n [\XX \cap \SS_3]$.
\end{lem}

\begin{proof}
Set $\UU  = \{ \delta_{i,j} \in \XX \mid (i,j) \not \in \{ (1,2), (2,1) \} \}$.
We first prove that $\KB_n [\XX]^{s_1} \subset \KB_n[\UU]$.
If $\XX \subset \UU$ there is nothing to prove.
We can therefore suppose that $\XX \not\subset \UU$.
Since $\XX$ is invariant under the action of $s_1$, we have $\XX = \UU \cup \{ \delta_{1,2}, \delta_{2,1} \}$. 
Set $\UU' = \UU \cup \{ \delta_{1,2} \}$ and $\UU'' = \UU \cup \{ \delta_{2,1} \}$.
By Lemma \ref{lem3_3}, $\KB_n [\XX] = \KB_n[\UU']*_{\KB_n[\UU]} \KB_n[\UU'']$.
Moreover, $s_1 (\KB_n[\UU']) = \KB_n [\UU'']$ and $s_1 (\KB_n [\UU'']) = \KB_n[\UU']$.
So, by Lemma \ref{lem3_6}, $\KB_n[\XX]^{s_1} \subset \KB_n [\UU]$.

For $2 \le k \le n$ we set $\VV_k = \{ \delta_{i,j} \in \XX \mid (i,j) \not\in (\{1,2\} \times \{1,2, \dots, k \}) \}$.
We show by induction on $k$ that $\KB_n [\XX]^{s_1} \subset \KB_n[\VV_k]$.
Since $\VV_2 = \UU$ the case $k=2$ holds.
Suppose that $k \ge 3$ and that the inductive hypothesis holds, that is, $\KB_n [\XX]^{s_1} \subset \KB_n [\VV_{k-1}]$.
If $\VV_k = \VV_{k-1}$ there is nothing to prove.
We can therefore suppose that $\VV_k \neq \VV_{k-1}$.
Since $\VV_{k-1}$ is invariant under the action of $s_1$ we have $\VV_{k-1} = \VV_k \cup \{\delta_{1,k}, \delta_{2,k}\}$.
Set $\VV_k' = \VV_k \cup \{ \delta_{1,k} \}$ and $\VV_k'' = \VV_k \cup \{ \delta_{2,k} \}$.
By Lemma \ref{lem3_3}, $\KB_n [\VV_{k-1}] = \KB_n [\VV_k'] *_{\KB_n [\VV_k]} \KB_n [\VV_k'']$.
Moreover, $s_1 (\KB_n [\VV_k']) = \KB_n [\VV_k'']$ and $s_1 (\KB_n [\VV_k'']) = \KB_n [\VV_k']$.
So, by Lemma \ref{lem3_6}, $\KB_n [\XX]^{s_1} \subset \KB_n [\VV_k]$.

For $2 \le k \le n$ we set $\WW_k = \{ \delta_{i,j} \in \XX \mid (i,j) \not \in ( \{ 1,2\} \times \{1,2, \dots, n\}) \text{ and } (i,j) \not\in (\{1, 2, \dots, k\} \times \{ 1,2\}) \}$.
We show by induction on $k$ that $\KB_n [\XX]^{s_1} \subset \KB_n[\WW_k]$.
Since $\WW_2 = \VV_n$ the case  $k=2$ holds.
Suppose that $k \ge 3$ and that the inductive hypothesis holds, that is, $\KB_n [\XX]^{s_1} \subset \KB_n [\WW_{k-1}]$.
If $\WW_k = \WW_{k-1}$ there is nothing to prove.
We can therefore suppose that $\WW_k \neq \WW_{k-1}$.
Since $\WW_{k-1}$ is invariant under the action of $s_1$ we have $\WW_{k-1} = \WW_k \cup \{ \delta_{k,1}, \delta_{k,2} \}$. 
Set $\WW_k' = \WW_k \cup \{ \delta_{k,1} \}$ and $\WW_k'' = \WW_k \cup \{ \delta_{k,2} \}$.
By Lemma \ref{lem3_3}, $\KB_n [\WW_{k-1}] = \KB_n [\WW_k'] *_{\KB_n [\WW_k]} \KB_n [\WW_k'']$.
Moreover, $s_1 (\KB_n [\WW_k']) = \KB_n [\WW_k'']$ and $s_1 (\KB_n [\WW_k'']) = \KB_n [\WW_k']$.
So, by Lemma \ref{lem3_6}, $\KB_n [\XX]^{s_1} \subset \KB_n [\WW_k]$.
The inclusion $\KB_n [\XX]^{s_1} \subset \KB_n [\XX \cap \SS_3]$ therefore follows from the fact that $\WW_n = \XX \cap \SS_3$.
The reverse inclusion $\KB_n[\XX \cap \SS_3] \subset \KB_n[\XX]^{s_1}$ is obvious.
\end{proof}

By applying the action of the symmetric group, from Lemma \ref{lem3_7} it immediately follows:

\begin{corl}\label{corl3_7b}
Let $k \in \{1, \dots, n-1\}$, let $\XX$ be a subset of $\SS$ invariant under the action of $s_k$, and let $\UU_k = \{ \delta_{i,j} \in \SS \mid i,j \not\in \{k, k+1\}\}$.
Then $\KB_n [\XX]^{s_k} = \KB_n [\XX \cap \UU_k]$.
\end{corl}

\begin{lem}\label{lem3_8}
Let $G_1, G_2, H$ be three groups.
We suppose that $H$ is a common subgroup of $G_1$ and $G_2$ and we set $G = G_1*_{H} G_2$.
Let $\tau : G \to G$ be an automorphism of order $2$ such that $\tau (G_1) = G_2$ and $\tau (G_2) = G_1$. 
Let  $\alpha \in G$ such that $\tau (\alpha) = \alpha^{-1}$.
Then there exist $\alpha' \in G$ and $\beta' \in H$ such that $\tau (\beta') = \beta'^{-1}$ and $\alpha = \alpha' \beta'\, \tau(\alpha'^{-1})$.
\end{lem}

\begin{proof}
Since $\tau (G_1) = G_2$ and $\tau (G_2) = G_1$, we have $\tau (H) = H$.
Let $T_1$ be a transversal of $H$ in $G_1$.
Then $T_2 = \tau (T_1)$ is a transversal of $H$ in $G_2$.
Let $\alpha \in G$ such that $\tau (\alpha) = \alpha^{-1}$.
Let $\alpha = \theta_1 \theta_2 \cdots \theta_\ell \beta$ be the normal form of $\alpha$.
We prove by induction on $\ell$ that there exist $\alpha' \in G$ and $\beta' \in H$ such that $\tau (\beta') = \beta'^{-1}$ and $\alpha = \alpha' \beta'\, \tau(\alpha'^{-1})$.
The case $\ell = 0$ being trivial we can suppose that $\ell \ge 1$ and that the inductive hypothesis holds. We have $1 = \alpha\, \tau(\alpha) = \theta_1 \cdots \theta_\ell \beta\,\tau(\theta_1) \cdots \tau(\theta_\ell)\, \tau(\beta)$.
By Proposition \ref{prop3_5} we must have $\theta_\ell \beta\, \tau(\theta_1) \in H$, namely, there exists $\beta_1 \in H$ such that $\theta_\ell \beta = \beta_1\,\tau(\theta_1)^{-1}$.
Note that this inclusion implies that $\ell$ is even (and therefore $\ell \ge 2$), since we should have either $\theta_\ell, \tau (\theta_1) \in G_1 \setminus H$ or $\theta_\ell, \tau (\theta_1) \in G_2 \setminus H$.
Thus, $ \alpha = \theta_1 \theta_2 \cdots \theta_{\ell-1} \beta_1\, \tau(\theta_1^{-1}) = \theta_1 \alpha_1\, \tau(\theta_1^{-1})$, where $\alpha_1 = \theta_2 \cdots \theta_{\ell-1} \beta_1$.
We have $\tau (\alpha_1) = \alpha_1^{-1}$ since $\tau (\alpha) = \alpha^{-1}$, so, by induction, there exist $\alpha_1' \in G$ and $\beta' \in H$ such that $\tau (\beta') = \beta'^{-1}$ and $\alpha_1 = \alpha_1' \beta'\, \tau(\alpha_1'^{-1})$.
We set $\alpha' = \theta_1 \alpha_1'$.
Then $\tau (\beta') = \beta'^{-1}$ and $\alpha = \alpha' \beta'\, \tau(\alpha'^{-1})$.
\end{proof}

\begin{lem}\label{lem3_9}
Let $\XX$ be a subset of $\SS$ invariant under the action of $s_1$.
Let $\alpha \in \KB_n [\XX]$ such that $s_1 (\alpha) = \alpha^{-1}$.
Then there exists $\alpha' \in \KB_n[\XX]$ such that $\alpha = \alpha'\, s_1 (\alpha'^{-1})$.
\end{lem}

\begin{proof}
Let $\alpha \in \KB_n [\XX]$ such that $s_1(\alpha) = \alpha^{-1}$.
Let $\UU = \{ \delta_{i,j} \in \XX \mid (i,j) \not \in \{(1,2), (2,1)\}\}$.
We show that there exist $\alpha' \in \KB_n [\XX]$ and $\beta' \in \KB_n [\UU]$ such that $s_1 (\beta') = \beta'^{-1}$ and $\alpha = \alpha' \beta' \, s_1 (\alpha'^{-1})$.
If $\XX = \UU$ there is nothing to prove. 
We can therefore suppose that $\XX \neq \UU$.
Since $\XX$ is invariant under the action of $s_1$ we have $\XX = \UU \cup \{ \delta_{1,2}, \delta_{2,1} \}$. 
Set $\UU' = \UU \cup \{ \delta_{1,2} \}$ and $\UU'' = \UU \cup \{ \delta_{2,1} \}$.
By Lemma \ref{lem3_3}, $KB_n [\XX] = \KB_n [\UU'] *_{\KB_n[\UU]} \KB_n [\UU'']$.
Moreover, $s_1 (\KB_n [\UU']) = \KB_n [\UU'']$ and $s_1 (\KB_n [\UU'']) = \KB_n [\UU']$.
So, by Lemma \ref{lem3_8}, there exist $\alpha' \in \KB_n [\XX]$ and $\beta' \in \KB_n [\UU]$ such that $s_1 (\beta') = \beta'^{-1}$ and $\alpha = \alpha' \beta' \, s_1 (\alpha'^{-1})$.

For $2 \le k \le n$ we set $\VV_k = \{ \delta_{i,j} \in \XX \mid (i,j) \not \in (\{ 1, 2\} \times \{1, 2, \dots, k\}) \}$.
We show by induction on  $k$ that there exist $\alpha' \in \KB_n [\XX]$ and $\beta' \in \KB_n [\VV_k]$ such that $s_1 (\beta') = \beta'^{-1}$ and $\alpha = \alpha' \beta'\, s_1 (\alpha'^{-1})$.
Since  $\VV_2 = \UU$ the case  $k=2$ follows from the previous paragraph.
Suppose that $k \ge 3$ and that the inductive hypothesis holds. 
So, there exist $\alpha_1' \in \KB_n [\XX]$ and $\beta_1' \in \KB_n [\VV_{k-1}]$ such that $s_1 (\beta_1') = \beta_1'^{-1}$ and $\alpha = \alpha_1' \beta_1'\, s_1(\alpha_1'^{-1})$.
If $\VV_k = \VV_{k-1}$ there is nothing to prove.
Thus, we can suppose that  $\VV_k \neq \VV_{k-1}$.
Since $\VV_{k-1}$ is invariant under the action of $s_1$ we have $\VV_{k-1} = \VV_k \cup \{ \delta_{1,k}, \delta_{2,k} \}$.
Set $\VV_k' = \VV_k \cup \{ \delta_{1,k} \}$ and $\VV_k'' = \VV_k \cup \{ \delta_{2,k} \}$.
By Lemma \ref{lem3_3}, $\KB_n [\VV_{k-1}] = \KB_n [\VV_k'] *_{\KB_n [\VV_k]} \KB_n [\VV_k'']$.
Moreover, $s_1 (\KB_n [\VV_k']) = \KB_n [\VV_k'']$ and $s_1 (\KB_n [\VV_k'']) = \KB_n [\VV_k']$.
By Lemma \ref{lem3_8} it follows that there exist $\alpha_2' \in \KB_n [\VV_{k-1}]$ and $\beta' \in \KB_n [\VV_k]$ such that $s_1 (\beta') = \beta'^{-1}$ and  $\beta_1' = \alpha_2' \beta' \, s_1 (\alpha_2'^{-1})$.
Set $\alpha' = \alpha_1' \alpha_2'$.
Then $\beta' \in \KB_n [\VV_k]$, $s_1 (\beta') = \beta'^{-1}$ and $\alpha = \alpha' \beta'\, s_1 (\alpha'^{-1})$.

For $2 \le k \le n$ we set $\WW_k = \{ \delta_{i,j} \in \XX \mid (i,j) \not\in (\{1,2\} \times \{1,2, \dots, n\}) \text{ and } (i,j) \not\in (\{ 1,2, \dots, k\} \times \{1,2\}) \}$.
We show by induction on $k$ that there exist $\alpha' \in \KB_n [\XX]$ and $\beta' \in \KB_n [\WW_k]$ such that $s_1 (\beta') = \beta'^{-1}$ and $\alpha = \alpha' \beta'\,s_1 (\alpha'^{-1})$.
Since $\WW_2 = \VV_n$ the case $k=2$ follows from the previous paragraph.
Suppose that $k \ge 3$ and that the inductive hypothesis holds. 
So, there exist $\alpha_1' \in \KB_n [\XX]$ and $\beta_1' \in \KB_n [\WW_{k-1}]$ such that $s_1 (\beta_1') = \beta_1'^{-1}$ and $\alpha = \alpha_1' \beta_1'\, s_1 (\alpha_1'^{-1})$.
If $\WW_k = \WW_{k-1}$ there is nothing to prove.
We can therefore suppose that $\WW_k \neq \WW_{k-1}$.
Since $\WW_{k-1}$ is invariant under the action of $s_1$ we have $\WW_{k-1} = \WW_k \cup \{ \delta_{k,1}, \delta_{k,2} \}$.
Set $\WW_k' = \WW_k \cup \{ \delta_{k,1} \}$ and $\WW_k'' = \WW_k \cup \{ \delta_{k,2} \}$.
By Lemma \ref{lem3_3}, $\KB_n [\WW_{k-1}] = \KB_n [\WW_k'] *_{\KB_n [\WW_k]} \KB_n[\WW_k'']$.
Moreover, $s_1 (\KB_n [\WW_k']) = \KB_n [\WW_k'']$ and $s_1 (\KB_n [\WW_k'']) = \KB_n [\WW_k']$.
By Lemma \ref{lem3_8} it follows that there exist $\alpha_2' \in \KB_n [\WW_{k-1}]$ and $\beta' \in \KB_n [\WW_k]$ such that $s_1 (\beta') = \beta'^{-1}$ and $\beta_1' = \alpha_2' \beta'\, s_1 (\alpha_2'^{-1})$.
Set $\alpha' = \alpha_1' \alpha_2'$.
Then  $\beta' \in \KB_n [\WW_k]$, $s_1 (\beta') = \beta'^{-1}$ and $\alpha = \alpha' \beta'\, s_1( \alpha'^{-1})$.

Notice that $\WW_n = \XX \cap \SS_3$.
Recall that $s_1 (\beta') = \beta'$ for all $\beta' \in \KB_n [\XX \cap \SS_3] = \KB_n [\WW_n]$ (see Lemma  \ref{lem3_7}).
By the above there exist $\alpha' \in \KB_n [\XX]$ and $\beta' \in \KB_n [\WW_n]$ such that $s_1 (\beta') = \beta'^{-1}$ and $\alpha = \alpha' \beta' \, s_1 (\alpha'^{-1})$.
So, $\beta' = s_1(\beta') = \beta'^{-1}$, hence $\beta'^2 = 1$, and therefore $\beta' = 1$, since, by Godelle--Paris \cite{GodPar1}, $\KB_n$ is torsion free. 
So, $\alpha = \alpha'\, s_1 (\alpha'^{-1})$.
\end{proof}

\begin{lem}\label{lem3_10}
Let $G_1, G_2, \dots, G_p, H$ be a collection of groups.
We suppose that $H$ is a subgroup of $G_j$ for all $j \in \{1, \dots, p\}$ and we consider the amalgamated product $G = G_1 *_H G_2 *_H \cdots *_H G_p$.
Let $\tau_1, \tau_2 : G \to G$ be two automorphisms satisfying the following properties:
\begin{itemize}
\item[(a)]
$\tau_1^2 = \tau_2^2 = 1$ and $\tau_1 \tau_2 \tau_1 = \tau_2 \tau_1 \tau_2$.
\item[(b)]
For all $i \in \{1,2\}$ and $j \in \{1, \dots, p\}$ there exists $k \in \{1, \dots, p\}$ such that $\tau_i (G_j) = G_k$.
\item[(c)]
For all $j \in \{1, \dots, p\}$ there exists $i \in \{1,2\}$ such that $\tau_i(G_j) \neq G_j$.
\item[(d)]
For all $i \in \{1, 2 \}$ we have $\tau_i (H) = H$.
\item[(e)]
For all $i \in \{1, 2\}$ and $\gamma \in H$ such that $\tau_i (\gamma) = \gamma^{-1}$ there exists $\delta \in H$ such that $\gamma = \delta\, \tau_i (\delta^{-1})$.
\end{itemize}
Let  $\alpha \in G$ satisfying the following equation:
\begin{equation} \label{eq:tresse}
\alpha \, \tau_2 (\alpha^{-1})\, (\tau_2 \tau_1) (\alpha)\, (\tau_2 \tau_1 \tau_2) (\alpha^{-1})\, (\tau_1 \tau_2) (\alpha)\, \tau_1 (\alpha^{-1}) = 1\,.
\end{equation}
Then there exist $\alpha', \alpha'' \in G$ and $\beta \in H$ such that $\alpha = \alpha' \beta \alpha''$, $\tau_1 (\alpha') = \alpha'$, $\tau_2 (\alpha'') = \alpha''$ and $\beta$ satisfies Equation (\ref{eq:tresse}).
\end{lem}

\begin{proof}
Let $\alpha \in G$ satisfying Equation (\ref{eq:tresse}).
It is easily checked that, if $\alpha$ is written $\alpha = \alpha' \beta \alpha''$, where $\tau_1 (\alpha') = \alpha'$ and $\tau_2 (\alpha'') = \alpha''$, then $\beta$ also satisfies Equation (\ref{eq:tresse}).
So, it suffices to show that there exist $\alpha', \alpha'' \in G$ and $\beta \in H$ such that $\alpha = \alpha' \beta \alpha''$, $\tau_1 (\alpha') = \alpha'$ and $\tau_2 (\alpha'') = \alpha''$.
If $\alpha \in H$ there is nothing to prove.
We can therefore suppose that $\alpha \not\in H$.
We write  $\alpha$ in the form $\alpha = \alpha_1 \alpha_2 \cdots \alpha_\ell$ such that
\begin{itemize}
\item[(a)]
for all $i \in \{1, \dots, \ell\}$ there exists $j = j(i) \in \{1, \dots, p\}$ such that $\alpha_i \in G_j \setminus H$,
\item[(b)]
$j(i) \neq j(i+1)$ for all $i \in \{1, \dots, \ell-1\}$.
\end{itemize}
We argue by induction on $\ell$.

Suppose first that $\ell = 1$.
We can assume without loss of generality that $\alpha \in G_1\setminus H$.
We set $\beta_1 = \alpha$, $\beta_2 = \tau_2 (\alpha^{-1})$, $\beta_3 = (\tau_2 \tau_1) (\alpha)$, $\beta_4 = (\tau_2 \tau_1 \tau_2)(\alpha^{-1})$, $\beta_5 = (\tau_1 \tau_2) (\alpha)$, and $\beta_6 = \tau_1 (\alpha^{-1})$.
By hypothesis we have $\beta_1 \beta_2 \beta_3 \beta_4 \beta_5 \beta_6 = 1$ and, for each $i \in \{1, \dots, 6\}$, there exists $j = j(i) \in \{1, \dots, p\}$ such that $\beta_i \in G_j \setminus H$.
If we had $\tau_1 (G_1) \neq G_1$ and $\tau_2 (G_1) \neq G_1$, then we would have $j(i) \neq j(i+1)$ for all $i \in \{1, \dots, 5\}$, thus, by Proposition \ref{prop3_5}, we would have $\beta_1 \beta_2 \beta_3 \beta_4 \beta_5 \beta_6 \neq 1$: contradiction. 
So, either $\tau_1 (G_1) = G_1$ or $\tau_2 (G_1) = G_1$.
On the other hand, by Condition (c) in the statement of the lemma, either $\tau_1 (G_1) \neq G_1$ or $\tau_2 (G_1) \neq G_1$.
Thus, either $\tau_1 (G_1) = G_1$ and $\tau_2 (G_1) \neq G_1$, or $\tau_1 (G_1) \neq G_1$ and $\tau_2 (G_1) = G_1$.
We assume that $\tau_1 (G_1) = G_1$ and $\tau_2 (G_1) \neq G_1$.
The case $\tau_1 (G_1) \neq G_1$ and $\tau_2 (G_1) = G_1$ is proved in a similar way.
So, $j(1) \neq j(2)$, $j(2) = j(3)$, $j(3) \neq j(4)$, $j(4) = j(5)$, and $j(5) \neq j(6)$.
Since $\beta_1 \beta_2 \beta_3 \beta_4 \beta_5 \beta_6 = 1$, by Proposition \ref{prop3_5}, either $\beta_2 \beta_3 \in H$, or $\beta_4 \beta_5 \in H$, that is, $\alpha^{-1} \, \tau_1 (\alpha) = \gamma \in H$.
Now, we have $\tau_1 (\gamma) = \tau_1 (\alpha^{-1}) \alpha = \gamma^{-1}$, hence, by Condition (e) in the statement of the lemma, there exists $\delta \in H$ such that $\gamma = \delta \, \tau_1 (\delta^{-1})$.
We set  $\alpha' = \alpha \delta$, $\alpha'' = 1$ and $\beta = \delta^{-1}$.
Then $\alpha = \alpha' \beta \alpha''$,  $\tau_1 (\alpha') = \alpha'$, $\tau_2 (\alpha'') = \alpha''$ and $\beta \in H$. 

Suppose that $\ell \ge 2$ and that the inductive hypothesis holds. 
It follows from Proposition \ref{prop3_5} and Equation (\ref{eq:tresse}) that either $\alpha_\ell \, \tau_2 (\alpha_\ell^{-1}) \in H$ or $\tau_1 (\alpha_1^{-1})\, \alpha_1 \in H$.
Suppose that $\alpha_\ell\, \tau_2(\alpha_\ell^{-1}) \in H$.
Then there exists $\gamma \in H$ such that $\alpha_\ell = \gamma\,\tau_2(\alpha_\ell)$.
By applying $\tau_2$ to this equality we obtain $\tau_2(\alpha_\ell) = \tau_2(\gamma)\,\alpha_\ell$, and therefore $\tau_2(\gamma)= \gamma^{-1}$.
By hypothesis there exists $\delta \in H$ such that $\gamma = \delta^{-1}\,\tau_2(\delta)$.
We set $ \alpha_\ell' = 1$, $\alpha_\ell'' = \delta \alpha_{\ell}$, $\alpha_{\ell-1}'' = \alpha_{\ell-1}\delta^{-1}$ and $\beta_1 = \alpha_1 \cdots \alpha_{\ell-2} \alpha_{\ell-1}''$.
We have $\alpha = \alpha_\ell' \beta_1 \alpha_\ell''$, $\tau_1 (\alpha_\ell') = \alpha_\ell'$ and $\tau_2 (\alpha_\ell'') = \alpha_\ell''$, hence $\beta_1$ satisfies Equation (\ref{eq:tresse}).
By induction, there exists $\gamma', \gamma'' \in G$ and $\beta \in H$ such that $\beta_1 = \gamma' \beta \gamma''$, $\tau_1 (\gamma') = \gamma'$ and $\tau_2 (\gamma'') = \gamma''$.
Set $\alpha' = \alpha_\ell' \gamma'$ and $\alpha'' = \gamma'' \alpha_\ell''$.
Then $\alpha = \alpha' \beta \alpha''$, $\tau_1 (\alpha') = \alpha'$ and $\tau_2 (\alpha'') = \alpha''$.
The case $\tau_1(\alpha_1^{-1}) \, \alpha_1 \in H$ can be proved in a similar way.
\end{proof}

\begin{lem}\label{lem3_11}
Let $\XX$ be a subset of $\SS$ invariant under the action of $s_1$ and under the action of $s_2$.
Let $\alpha \in \KB_n[\XX]$ such that
\begin{equation}
\alpha\, s_2(\alpha^{-1})\, (s_2 s_1)(\alpha)\, (s_2 s_1 s_2)(\alpha^{-1}) \, (s_1 s_2) (\alpha) \, s_1 (\alpha^{-1}) = 1\,.
\label{eq3_2}
\end{equation}
Then there exist $\alpha', \alpha'' \in \KB_n[\XX]$ such that  $s_1 (\alpha') = \alpha'$, $s_2 (\alpha'') = \alpha''$ and $\alpha = \alpha' \alpha''$.
\end{lem}

\begin{proof}
For $4 \le k \le n$ we set $\UU_k = \{\delta_{i,j} \in \XX \mid (i,j) \not \in (\{1,2,3\} \times \{4, \dots, k\})\}$.
We set also $\UU_3 = \XX$.
We show by induction on $k$ that there exist $\alpha', \alpha'' \in \KB_n[\XX]$ and $\beta \in \KB_n[\UU_k]$ such that $\alpha = \alpha' \beta \alpha''$, $s_1 (\alpha') = \alpha'$, $s_2 (\alpha'') = \alpha''$ and $\beta$ satisfies Equation (\ref{eq3_2}).
The case $k = 3$ is true by hypothesis.
Suppose that $4 \le k \le n$ and that the inductive hypothesis holds. 
So, there exist $\alpha_1', \alpha_1'' \in \KB_n[\XX]$ and $\beta_1 \in \KB_n[\UU_{k-1}]$ such that $\alpha = \alpha_1' \beta_1 \alpha_1''$, $s_1 (\alpha_1') = \alpha_1'$, $s_2 (\alpha_1'') = \alpha_1''$ and $\beta_1$ satisfies Equation (\ref{eq3_2}). 
If  $\UU_{k-1} = \UU_k$ there is nothing to prove.
Suppose that  $\UU_{k-1} \neq \UU_k$.
Since $\UU_{k-1}$ is invariant under the action of $s_1$ and under the action of $s_2$, we have $\UU_{k-1} = \UU_k \cup \{ \delta_{1,k}, \delta_{2,k}, \delta_{3,k} \}$. 
Set $G_j = \KB_n[\UU_k \cup \{\delta_{j,k}\}]$ for $j \in \{1,2,3\}$ and $H = \KB_n[\UU_k]$.
By Lemma \ref{lem3_3}, $\KB_n[\UU_{k-1}] = G_1*_H G_2*_H G_3$.
Moreover, we have the following properties.
\begin{itemize}
\item
For each $i \in \{1,2\}$ and each $j \in \{1,2,3\}$ there exists $k \in \{1,2,3\}$ such that $s_i (G_j) = G_k$.
\item
For each $j\in \{1,2,3\}$ there exists $i \in \{1,2\}$ such that $s_i(G_j) \neq G_j$.
\item 
For each  $i \in \{1,2\}$ we have $s_i(H) = H$.
\item
By Lemma \ref{lem3_9}, for each $i \in \{1,2\}$ and each $\gamma \in H$ such that $s_i(\gamma) = \gamma^{-1}$, there exists $\delta \in H$ such that $\gamma = \delta  \, s_i (\delta^{-1})$.
\end{itemize}
By Lemma \ref{lem3_10} it follows that there exist $\alpha_2', \alpha_2'' \in \KB_n[\UU_{k-1}]$ and $\beta \in \KB_n[\UU_k]$ such that $\beta_1 = \alpha_2' \beta \alpha_2''$, $s_1 (\alpha_2') = \alpha_2'$, $s_2 (\alpha_2'') = \alpha_2''$ and $\beta$ satisfies Equation (\ref{eq3_2}).
We set $\alpha'= \alpha_1' \alpha_2'$ and $\alpha'' = \alpha_2'' \alpha_1''$.
Then $\alpha = \alpha' \beta \alpha''$, $s_1 (\alpha') = \alpha'$ and $s_2 (\alpha'') = \alpha''$.

For $4 \le k \le n$ we set $\VV_k = \{\delta_{i,j} \in \UU_n \mid (i,j) \not \in (\{4, \dots, k\} \times \{1,2,3\})\}$.
We set also $\VV_3 = \UU_n$.
We show by induction on $k$ that there exist $\alpha', \alpha'' \in \KB_n[\XX]$ and $\beta \in \KB_n[\VV_k]$ such that $\alpha = \alpha' \beta \alpha''$, $s_1 (\alpha') = \alpha'$, $s_2 (\alpha'') = \alpha''$ and $\beta$ satisfies Equation (\ref{eq3_2}).
The case $k = 3$ is true by the above. 
Suppose that $4 \le k \le n$ and that the inductive hypothesis holds. 
So, there exist $\alpha_1', \alpha_1'' \in \KB_n[\XX]$ and $\beta_1 \in \KB_n[\VV_{k-1}]$ such that $\alpha = \alpha_1' \beta_1 \alpha_1''$, $s_1 (\alpha_1') = \alpha_1'$, $s_2 (\alpha_1'') = \alpha_1''$ and $\beta_1$ satisfies Equation (\ref{eq3_2}).
If $\VV_{k-1} = \VV_k$ there is nothing to prove.
Suppose that $\VV_{k-1} \neq \VV_k$.
Since $\VV_{k-1}$ is invariant under the action of $s_1$ and under the action of $s_2$, we have $\VV_{k-1} = \VV_k \cup \{ \delta_{k,1}, \delta_{k,2}, \delta_{k,3} \}$.
Set $G_j = \KB_n[\VV_k \cup \{\delta_{k,j}\}]$ for  $j \in \{1,2,3\}$ and $H = \KB_n[\VV_k]$.
By Lemma \ref{lem3_3}, $\KB_n[\VV_{k-1}] = G_1*_H G_2*_H G_3$.
Moreover, we have the following properties.
\begin{itemize}
\item
For each $i \in \{1,2\}$ and each $j \in \{1,2,3\}$ there exists $k \in \{1,2,3\}$ such that $s_i (G_j) = G_k$.
\item
For each $j\in \{1,2,3\}$ there exists $i \in \{1,2\}$ such that $s_i(G_j) \neq G_j$.
\item 
For each $i \in \{1,2\}$ we have $s_i(H) = H$.
\item
By Lemma \ref{lem3_9}, for each $i \in \{1,2\}$ and each $\gamma \in H$ such that $s_i(\gamma) = \gamma^{-1}$, there exists $\delta \in H$ such that $\gamma = \delta  \, s_i (\delta^{-1})$.
\end{itemize}
By Lemma \ref{lem3_10} it follows that there exist $\alpha_2', \alpha_2'' \in \KB_n[\VV_{k-1}]$ and $\beta \in \KB_n[\VV_k]$ such that $\beta_1 = \alpha_2' \beta \alpha_2''$, $s_1 (\alpha_2') = \alpha_2'$, $s_2 (\alpha_2'') = \alpha_2''$ and $\beta$ satisfies Equation (\ref{eq3_2}).
We set $\alpha'= \alpha_1' \alpha_2'$ and $\alpha'' = \alpha_2'' \alpha_1''$.
Then $\alpha = \alpha' \beta \alpha''$, $s_1 (\alpha') = \alpha'$ and $s_2 (\alpha'') = \alpha''$.

Set $\WW_1 = \{ \delta_{i,j} \in \XX \mid (i,j) \in (\{1,2,3\} \times \{1,2,3\}) \}$ and $\WW_2 = \{ \delta_{i,j} \in \XX \mid (i,j) \in (\{4, \dots, n\} \times \{4, \dots, n\}) \}$.
Notice that $\VV_n = \WW_1 \sqcup \WW_2$, $\KB_n [\VV_n] = \KB_n [\WW_1] \times \KB_n [\WW_2]$, and $s_1 (\gamma) = s_2 (\gamma) = \gamma$ for all $\gamma \in \KB_n [\WW_2]$.
By the above there exist $\alpha_1', \alpha_1'' \in \KB_n [\XX]$ and $\beta_1 \in \KB_n [\VV_n]$ such that $\alpha = \alpha_1' \beta_1 \alpha_1''$, $s_1 (\alpha_1') = \alpha_1'$, $s_2 (\alpha_1'') = \alpha_1''$ and 
$\beta_1$ satisfies Equation (\ref{eq3_2}).
Let $\beta \in \KB_n [\WW_1]$ and $\alpha_2'' \in \KB_n[\WW_2]$ such that $\beta_1 = \beta \alpha_2''$.
Set $\alpha' = \alpha_1'$ and $\alpha'' = \alpha_2'' \alpha_1''$.
Then $\alpha = \alpha' \beta \alpha''$, $s_1 (\alpha') = \alpha'$, $s_2 (\alpha'') = \alpha''$ and $\beta$ satisfies Equation (\ref{eq3_2}).

Let $\WW_{1,1} = \{ \delta_{i,j} \in \XX \mid (i,j) \in \{ (1,2),\, (2,3),\, (3,1)\} \}$ and $\WW_{1,2} = \{ \delta_{i,j} \in \XX \mid (i,j) \in \{ (2,1),\,(3,2),\, (1,3)\} \}$. 
Set $G_1 = \KB_n [\WW_{1,1}]$ and $G_2 = \KB_n [\WW_{1,2}]$.
By Lemma \ref{lem3_3}, $\KB_n [\WW_1] = G_1 * G_2$.
Moreover, $s_1 (G_1) = s_2 (G_1) = G_2$ and $s_1 (G_2) = s_2 (G_2) = G_1$.
From Lemma \ref{lem3_10} applied with $H = \{1\}$ it follows that there exist $\beta', \beta'' \in \KB_n [\WW_1]$ such that $\beta = \beta' \beta''$, $s_1 (\beta') = \beta'$ and $s_2 (\beta'') = \beta''$.
Actually, by Lemma \ref{lem3_7}, $\beta' = \beta'' = 1$, hence $\beta = 1$.
So, $\alpha = \alpha' \alpha''$, $s_1(\alpha') = \alpha'$ and $s_2 (\alpha'') = \alpha''$.
\end{proof}

As announced in Section \ref{sec2}, we take advantage of the results of the present section to prove Lemma \ref{lem2_6}.

\begin{proof}[Proof of Lemma \ref{lem2_6}]
Suppose instead that $\zeta_1$ is an inner automorphism, that is, $\zeta_1 = c_\gamma : \VB_n \to \VB_n$, $\delta \mapsto \gamma \delta \gamma^{-1}$, for some $\gamma \in \VB_n$.
We have $\gamma \neq 1$ since $\zeta_1 \neq \id$.
We write $\gamma$ in the form $\gamma = \alpha\, \iota(w)$ with $\alpha \in \KB_n$ and $w \in \SSS_n$.
For each $i \in \{ 1, \dots, n-1 \}$ we have $s_i = \pi_K (\tau_i) = \pi_K (\zeta_1 (\tau_i)) = \pi_K(\gamma \tau_i \gamma^{-1}) = w s_i w^{-1}$, hence $w$ lies in the center of $\SSS_n$ which is trivial, and therefore $w=1$ and $\gamma= \alpha \in \KB_n$.

Note that $\zeta_1 (\delta_{i,j}) = \delta_{j,i}$ for all $i,j \in \{1, \dots, n \}$, $i \neq j$.
Take $i,j \in \{1, \dots, n\}$, $i<j$, and set $\UU_{i,j} =\SS \setminus \{ \delta_{i,j}, \delta_{j,i} \}$, $\UU_{i,j}' = \UU_{i,j} \cup \{ \delta_{i,j} \}$ and $\UU_{i,j}'' = \UU_{i,j} \cup \{ \delta_{j,i} \}$.
By Lemma \ref{lem3_3}, $\KB_n = \KB_n [\UU_{i,j}']*_{\KB_n [\UU_{i,j}]} \KB_n [\UU_{i,j}'']$.
Moreover, $\zeta_1 (\KB_n [\UU_{i,j}']) = \KB_n [\UU_{i,j}'']$ and $\zeta_1 (\KB_n [\UU_{i,j}'']) = \KB_n [\UU_{i,j}']$.
From Lemma \ref{lem3_6} it follows that $\KB_n^{\zeta_1} \subset \KB_n [\UU_{i,j}]$.
Since $\bigcap_{1 \le i < j \le n} \UU_{i,j} = \emptyset$, by Theorem \ref{thm3_2}, $\bigcap_{1 \le i < j \le n} \KB_n [\UU_{i,j}] = \KB_n[ \emptyset] = \{ 1 \}$, thus $\KB_n^{\zeta_1} = \{ 1 \}$.
But $\alpha \in \KB_n^{\zeta_1}$ and $\alpha = \gamma \neq 1$, which is a contradiction.  
So, $\zeta_1$ is not an inner automorphism. 
\end{proof}


\section{From virtual braid groups to symmetric groups}\label{sec4}

The following is well-known and can be easily deduced from Artin \cite{Artin1} and Lin \cite{Lin1, Lin2}.
It is a preliminary to the proof of Theorem \ref{thm2_1}.

\begin{prop}\label{prop4_1}
Let $n, m \in \N$ such that $n \ge 5$, $m \ge 2$ and $n \ge m$, and let $\varphi : \SSS_n \to \SSS_m$ be a homomorphism.
Then, up to conjugation, one of the following possibilities holds. 
\begin{itemize}
\item[(1)]
$\varphi$ is Abelian,
\item[(2)]
$n = m$ and $\varphi = \id$,
\item[(3)]
$n = m = 6$ and $\varphi = \nu_6$.
\end{itemize}
\end{prop}

\begin{proof}[Proof of Theorem \ref{thm2_1}]
Let $n, m \in \N$ such that $n \ge 5$, $m \ge 2$ and $n \ge m$, and let $\psi : \VB_n \to \SSS_m$ be a homomorphism. 
By Proposition \ref{prop4_1} one of the following possibilities holds up to conjugation. 
\begin{itemize}
\item
$\psi \circ \iota$ is Abelian,
\item
$n = m$ and $\psi \circ \iota = \id$,
\item
$n = m = 6$ and $\psi \circ \iota = \nu_6$.
\end{itemize}

Suppose that $\psi \circ \iota$ is Abelian. 
Then there exists $w_1 \in \SSS_m$ such that $w_1 = (\psi \circ \iota) (s_i) = \psi (\tau_i)$ for all $i \in \{ 1, \dots, n-1 \}$.
Since $s_i^2 = 1$, we have $w_1^2 = 1$.
Set $w_2 = \psi (\sigma_1)$.
From the relation $\tau_i \tau_{i+1} \sigma_i = \sigma_{i+1} \tau_i \tau_{i+1}$ it follows that $\psi (\sigma_i) = w_1^2\,\psi (\sigma_i) = \psi (\sigma_{i+1}) w_1^2 = \psi (\sigma_{i+1})$ for all $i \in \{ 1, \dots, n-2 \}$, hence $\psi (\sigma_i) = w_2$ for all $i \in \{1, \dots, n-1 \}$.
Finally, from the relation $\tau_1 \sigma_3 = \sigma_3 \tau_1$ it follows that $w_1 w_2 = w_2 w_1$.
So, $\psi$ is Abelian.

Suppose that $n = m$ and $\psi \circ \iota = \id$.
Then $\psi (\tau_i) = s_i$ for all $i \in \{1, \dots, n-1\}$.
From the relations $\sigma_1 \tau_i = \tau_i \sigma_1$, $3 \le i \le n-1$, it follows that $\psi (\sigma_1)$ lies in the centralizer of $\langle s_3, \dots, s_{n-1} \rangle$ in $\SSS_n$, which is equal to $\langle s_1 \rangle = \{ 1, s_1 \}$, hence either $\psi (\sigma_1) = 1$ or $\psi (\sigma_1) = s_1$.
If $\psi (\sigma_1) = 1$, then $\psi (\sigma_i) = 1$, since $\sigma_i$ is conjugate to $\sigma_1$ in $\VB_n$, for all $i \in \{1, \dots, n-1 \}$, hence $\psi = \pi_K$.
Assume that $\psi (\sigma_1) = s_1$.
We show by induction on $i$ that $\psi (\sigma_i) = s_i$ for all $i \in \{1, \dots, n-1 \}$.
The case $i=1$ is true by hypothesis. 
Suppose that $i \ge 2$ and $\psi (\sigma_{i-1}) = s_{i-1}$.
Then, since $\tau_{i-1} \tau_i \sigma_{i-1} = \sigma_i \tau_{i-1} \tau_i$, we have $\psi (\sigma_i) = s_{i-1} s_i s_{i-1} s_i s_{i-1} = s_i$.
So, $\psi (\tau_i) = \psi (\sigma_i) = s_i$ for all $i \in \{1, \dots, n-1\}$, that is, $\psi = \pi_P$.

Suppose that $n = m = 6$ and $\psi \circ \iota = \nu_6$.
Then $\nu_6^{-1} \circ \psi \circ \iota = \id$, hence, by the above, either $\nu_6^{-1} \circ \psi = \pi_K$ or $\nu_6^{-1} \circ \psi = \pi_P$, and therefore either $\psi = \nu_6 \circ \pi_K$ or $\psi = \nu_6 \circ \pi_P$.
\end{proof}


\section{From symmetric groups to virtual braid groups}\label{sec5}

The core of the proof of Theorem \ref{thm2_2} lies in the following lemma.

\begin{lem}\label{lem5_1}
Let $n \in \N$, $n \ge 3$, and let $\varphi : \SSS_n \to \VB_n$ be a homomorphism such that $\pi_K \circ \varphi = \id$.
Then $\varphi$ is conjugate to $\iota$.
\end{lem}

\begin{proof}
Since $\pi_K \circ \varphi = \id$, for each $i \in \{1, \dots, n-1\}$ there exists $\alpha_i \in \KB_n$ such that $\varphi (s_i) = \alpha_i \tau_i$.
We prove by induction on $k$ that there exists a homomorphism $\varphi' : \SSS_n \to \VB_n$ conjugate to $\varphi$ such that $\pi_K \circ \varphi' = \id$ and $\varphi' (s_i) = \tau_i$ for all $i \in \{1, \dots, k \}$.
The case $k=n-1$ ends the proof of the lemma.

Suppose that $k = 1$.
We have $1 = \varphi (s_1)^2 = \alpha_1\, s_1(\alpha_1)\, \tau_1^2 = \alpha_1\, s_1 (\alpha_1)$, hence $s_1 (\alpha_1) = \alpha_1^{-1}$.
By Lemma \ref{lem3_9} there exists $\beta_1 \in \KB_n$ such that $\alpha_1 = \beta_1\, s_1(\beta_1^{-1})$.
Thus, $\varphi (s_1) = \alpha_1 \tau_1 = \beta_1\, s_1(\beta_1^{-1})\, \tau_1 = \beta_1 \tau_1 \beta_1^{-1}$. 
Set $\varphi' = c_{\beta_1^{-1}} \circ \varphi$. 
Then $\varphi'$ is conjugate to $\varphi$, $\pi_K \circ \varphi' = \id$, since $\beta_1 \in \KB_n$, and $\varphi' (s_1) = \tau_1$.

We assume that $k=2$ and $\varphi (s_1) = \tau_1$.
Since $\varphi (s_1 s_2 s_1) = \varphi (s_2 s_1 s_2)$, we have $\tau_1 \alpha_2 \tau_2 \tau_1 = \alpha_2 \tau_2 \tau_1 \alpha_2 \tau_2$, hence $s_1 (\alpha_2) = \alpha_2\, (s_2 s_1) (\alpha_2)$.
On the other hand, as in the previous paragraph, from the equality $\varphi (s_2)^2 = 1$ it follows that there exists $\beta_2 \in \KB_n$ such that $\alpha_2 = \beta_2 \, s_2 (\beta_2^{-1})$. 
Thus, 
\begin{gather*}
s_1 \big( \beta_2 \, s_2 (\beta_2^{-1}) \big) = \big( \beta_2\, s_2 (\beta_2^{-1}) \big) \, (s_2 s_1) \big( \beta_2 \, s_2 (\beta_2^{-1}) \big) \quad \Rightarrow\\
\beta_2 \, s_2 (\beta_2^{-1}) \, (s_2 s_1) (\beta_2) \, (s_2 s_1 s_2) (\beta_2^{-1}) \, (s_1 s_2) (\beta_2) \, s_1 (\beta_2^{-1}) = 1\,.
\end{gather*}
By Lemma \ref{lem3_11} there exist $\beta_2', \beta_2'' \in \KB_n$ such that $\beta_2 = \beta_2' \beta_2''$, $s_1 (\beta_2') = \beta_2'$ and $s_2 (\beta_2'') = \beta_2''$.
Thus, $\varphi(s_2) = \alpha_2 \tau_2 = \beta_2 \tau_2 \beta_2^{-1} = \beta_2' \tau_2 \beta_2'^{-1}$.
Set $\varphi' = c_{\beta_2'^{-1}} \circ \varphi$.
Then $\varphi'$ is conjugate to $\varphi$, $\pi_K \circ \varphi' = \id$, since $\beta_2' \in \KB_n$, $\varphi' (s_1) = \tau_1$, since $s_1 (\beta_2') = \beta_2'$, and $\varphi' (s_2) = \tau_2$ by construction.

We assume that $k \ge 3$ and $\varphi (s_i) = \tau_i$ for all $i \in \{1, \dots, k-1\}$.
Let $\ell \in \{1, \dots, k-2\}$.
Since $\varphi (s_k s_\ell) = \varphi(s_\ell s_k)$, we have $\alpha_k \tau_k \tau_\ell = \tau_\ell \alpha_k \tau_k$, hence $s_\ell (\alpha_k) = \alpha_k$.
By Corollary \ref{corl3_7b} it follows that $\alpha_k \in \KB_n [\UU_\ell]$, where $\UU_\ell = \{ \delta_{i,j} \in \SS \mid i,j \not\in \{\ell, \ell+1\} \}$.
Recall that $\SS_k = \{ \delta_{i,j} \in \SS \mid k \le i \neq j \le n \}$.
We have $\bigcap_{1 \le \ell \le k-2} \UU_\ell = \SS_k$ hence, by Theorem \ref{thm3_2}, $\alpha_k \in \KB_n [\SS_k]$.
Since $\varphi (s_{k-1} s_k s_{k-1}) = \varphi(s_k s_{k-1} s_k)$, we have $\tau_{k-1} \alpha_k \tau_k \tau_{k-1} = \alpha_k \tau_k \tau_{k-1} \alpha_k \tau_k$, hence $s_{k-1} (\alpha_k) = \alpha_k \, (s_k s_{k-1})(\alpha_k)$.
On the other hand, as in the two previous paragraphs, from the equality $\varphi(s_k)^2 = 1$ it follows that there exists $\beta_k \in \KB_n [\SS_k]$ such that $\alpha_k = \beta_k \, s_k (\beta_k^{-1})$.
So, 
\begin{gather*}
s_{k-1} \big( \beta_k\, s_k (\beta_k^{-1}) \big) = \big( \beta_k \, s_k (\beta_k^{-1}) \big) \, (s_k s_{k-1}) \big( \beta_k \, s_k (\beta_k^{-1}) \big) \quad \Rightarrow\\
\beta_k\, s_k (\beta_k^{-1}) \, (s_k s_{k-1}) (\beta_k) \, (s_k s_{k-1} s_k) (\beta_k^{-1}) \, (s_{k-1} s_k) (\beta_k) \, s_{k-1} (\beta_k^{-1}) = 1\,.
\end{gather*}
Note that $\SS_k$ is not invariant under the action of $s_{k-1}$, but $\SS_{k-1}$ is and $\SS_k \subset \SS_{k-1}$.
By Lemma \ref{lem3_11} there exist $\beta_k', \beta_k'' \in \KB_n[\SS_{k-1}]$ such that $\beta_k = \beta_k' \beta_k''$, $s_{k-1} (\beta_k') = \beta_k'$ and $s_k (\beta_k'') = \beta_k''$.
So, $\varphi (s_k) = \alpha_k \tau_k = \beta_k \tau_k \beta_k^{-1} = \beta_k' \tau_k \beta_k'^{-1}$.
Since $s_{k-1} (\beta_k') = \beta_k'$, by Corollary \ref{corl3_7b}, $\beta_k' \in \KB_n[\UU_{k-1}]$, where $\UU_{k-1} = \{ \delta_{i,j} \in \SS \mid i,j \not\in \{k-1, k\} \}$.
Since $\SS_{k-1} \cap \UU_{k-1} = \SS_{k+1}$, by Theorem \ref{thm3_2} it follows that $\beta_k' \in \KB_n [\SS_{k+1}]$, hence $s_i (\beta_k') = \beta_k'$ for all $i \in \{1, \dots, k-1 \}$.
Set $\varphi' = c_{\beta_k'^{-1}} \circ \varphi$.
Then $\varphi'$ is conjugate to $\varphi$, $\pi_K \circ \varphi' = \id$, and $\varphi' (s_i) = \tau_i$ for all $i \in \{1, \dots, k \}$.
\end{proof}

\begin{proof}[Proof of Theorem \ref{thm2_2}]
Let $n, m \in \N$ such that $n \ge 5$, $m \ge 2$, and $n \ge m$, and let $\varphi : \SSS_n \to \VB_m$ be a homomorphism. 
By Proposition \ref{prop4_1} one of the following possibilities holds up to conjugation. 
\begin{itemize}
\item
$\pi_K \circ \varphi$ is Abelian,
\item
$n = m$ and $\pi_K \circ \varphi = \id$,
\item
$n = m = 6$ and $\pi_K \circ \varphi = \nu_6$.
\end{itemize}

Assume that $\pi_K \circ \varphi$ is Abelian.
There exists $w \in \SSS_m$ such that $(\pi_K \circ \varphi) (s_i) = w$ for all $i \in \{1, \dots, n-1 \}$.
Set $\beta_0 = \iota (w) \in \VB_m$.
For each $i \in \{1, \dots, n-1\}$ there exists $\alpha_i \in \KB_m$ such that $\varphi (s_i) = \alpha_i \beta_0$.
Since $s_1^2 = 1$, we have $w^2 =1$, hence $\beta_0^2 = 1$.  
On the other hand, for each $i \in \{1, \dots, n-1\}$, we have $1 = \varphi(s_i)^2 = \alpha_i \beta_0 \alpha_i \beta_0$, hence $\alpha_i \beta_0 = \beta_0 \alpha_i^{-1}$.
Let $i \in \{2, \dots, n-1\}$.
We have $\varphi (s_1 s_i) = \alpha_1 \beta_0 \beta_0 \alpha_i^{-1} = \alpha_1 \alpha_i^{-1} \in \KB_m$, this element is of finite order, since $s_1 s_i$ is of finite order, and, by Godelle--Paris \cite{GodPar1}, $\KB_m$ is torsion free, hence $\alpha_1 \alpha_i^{-1} = 1$, that is, $\alpha_i = \alpha_1$.
So, $\varphi(s_i) = \alpha_1 \beta_0$ for all $i \in \{1, \dots, n-1\}$, hence $\varphi$ is Abelian. 

Suppose that $n = m$ and $\pi_K \circ \varphi = \id$.
Then, by Lemma \ref{lem5_1}, $\varphi$ is conjugate to $\iota$.

Suppose that $n = m = 6$ and $\pi_K \circ \varphi = \nu_6$.
We have $\pi_K \circ \varphi \circ \nu_6^{-1} = \id$ hence, by Lemma \ref{lem5_1}, $\varphi \circ \nu_6^{-1}$ is conjugate to $\iota$, that is, there exists $\alpha \in \VB_6$ such that $\varphi \circ \nu_6^{-1} = c_\alpha \circ \iota$.
Then $\varphi = c_\alpha \circ \iota \circ \nu_6$, hence $\varphi$ is conjugate to $\iota \circ \nu_6$.
\end{proof}


\section{From virtual braid groups to virtual braid groups}\label{sec6}

\begin{lem}\label{lem6_0}
Let $n \ge 3$, let $i,j,k \in \{1, \dots, n\}$ pairwise distinct, and let $\ell_1, \ell_2 \in \Z$.
Then $\delta_{i,j}^{\ell_1} \delta_{j,k}^{\ell_2} = 1$ if and only if $\ell_1 = \ell_2 = 0$.
Similarly, we have $\delta_{j,i}^{\ell_1} \delta_{k,j}^{\ell_2} =1$ if and only if $\ell_1 = \ell_2 = 0$.
\end{lem}

\begin{proof}
Suppose that $\delta_{i,j}^{\ell_1} \delta_{j,k}^{\ell_2} = 1$.
Set $\ell_1 = 2 t_1 + \varepsilon_1$ and $\ell_2 = 2 t_2 + \varepsilon_2$ where $t_1, t_2 \in \Z$ and $\varepsilon_1, \varepsilon_2 \in \{ 0, 1 \}$.
We have $\pi_P (\delta_{i,j}^{\ell_1} \delta_{j,k}^{\ell_2}) = (i,j)^{\varepsilon_1} (j,k)^{\varepsilon_2} = 1$, hence $\varepsilon_1 = \varepsilon_2 = 0$.
So, $(\delta_{i,j}^2)^{t_1} (\delta_{j,k}^2)^{t_2} = 1$.
By Crisp--Paris \cite{CriPar1} the subgroup of $\KB_n[ \{\delta_{i,j}, \delta_{j,k}\}]$ generated by $\{ \delta_{i,j}^2, \delta_{j,k}^2 \}$ is a free group freely generated by $\{ \delta_{i,j}^2, \delta_{j,k}^2 \}$, hence $t_1 = t_2 =0$.
We show in the same way that, if $\delta_{j,i}^{\ell_1} \delta_{k,j}^{\ell_2} = 1$, then $\ell_1 = \ell_2 = 0$.
It is clear that, if $\ell_1 = \ell_2 = 0$, then $\delta_{i,j}^{\ell_1} \delta_{j,k}^{\ell_2} = \delta_{j,i}^{\ell_1} \delta_{k,j}^{\ell_2} =1$.
\end{proof}

\begin{lem}\label{lem6_1}
Let $n = 6$.
Set $u_i = \nu_6 (s_i)$ for all $i \in \{1 ,2, 3, 4, 5\}$.
Let $H$ be the subgroup of $\SSS_6$ generated by $\{ u_3, u_4, u_5\}$.
Then $\KB_6^H = \{1\}$, where $\KB_6^H = \{ \alpha \in \KB_6 \mid w(\alpha) = \alpha \text{ for all } w \in H \}$.
\end{lem}

\begin{proof}
Let $U = \{ u_3, u_4, u_5, u_3 u_4 u_3, u_4 u_5 u_4, u_3 u_4 u_5 u_4 u_3 \}$.
We have 
\begin{gather*}
u_3 = (1, 3) (2, 4) (5, 6)\,,\
u_4 = (1, 2) (3, 5) (4, 6)\,,\
u_5 = (2, 3) (1, 4) (5, 6)\,,\\
u_3 u_4 u_3 = (1,6) (2,5) (3,4)\,,\
u_4 u_5 u_4 = (1,5) (2,6) (3,4)\,,\\
u_3 u_4 u_5 u_4 u_3 = (1,2) (3,6) (4,5)\,.
\end{gather*}
Let $1 \le i < j \le 6$.
Set $\UU_{i,j} = \SS \setminus \{ \delta_{i,j}, \delta_{j,i} \}$, $\UU_{i,j}' = \UU_{i,j} \cup \{ \delta_{i,j} \}$ and $\UU_{i,j}'' = \UU_{i,j} \cup \{ \delta_{j,i} \}$.
We have $\KB_6 = \KB_6 [\UU_{i,j}'] *_{\KB_6 [\UU_{i,j}]} \KB_6[\UU_{i,j}'']$ by Lemma \ref{lem3_3}.
On the other hand, it is easily observed that there exists $v \in U$ such that $(i,j)$ is a cycle in the decomposition of $v$ as a product of disjoint cycles. 
For that $v$ we have $v (\KB_6 [\UU_{i,j}']) = \KB_6 [\UU_{i,j}'']$ and $v (\KB_6 [\UU_{i,j}'']) = \KB_6 [\UU_{i,j}']$.
By Lemma \ref{lem3_6} we deduce that $\KB_6^H \subset \KB_6^v \subset \KB_6 [\UU_{i,j}]$.
We have $\bigcap_{1 \le i < j \le 6} \UU_{i,j} = \emptyset$, hence, by Theorem \ref{thm3_2}, $\KB_6^H \subset \KB_6[\emptyset] = \{1\}$.
\end{proof}

We will use the following notation in the proof of Theorem \ref{thm2_3}.
Let $F_2 = F_2 (x,y)$ be the free group of rank $2$ freely generated by $\{ x,y \}$.
If $G$ is a group, $\alpha, \beta \in G$, and $\omega \in F_2$, then $\omega (\alpha, \beta)$ denotes the image of $\omega$ under the homomorphism $F_2 \to G$ which sends $x$ to $\alpha$ and $y$ to $\beta$.

\begin{proof}[Proof of Theorem \ref{thm2_3}]
Let $n, m \in \N$ such that $n \ge 5$, $m \ge 2$ and $n \ge m$, and let $\psi : \VB_n \to \VB_m$ be a homomorphism. 
By Theorem \ref{thm2_2} one of the following possibilities holds up to conjugation. 
\begin{itemize}
\item
$\psi \circ \iota$ is Abelian,
\item
$n = m$ and $\psi \circ \iota = \iota$,
\item
$n = m = 6$ and $\psi \circ \iota = \iota \circ \nu_6$.
\end{itemize}

Assume that $\psi \circ \iota$ is Abelian. 
We argue in the same way as in the Abelian case in the proof of Theorem \ref{thm2_1}.
There exists $\beta_1 \in \VB_m$ such that $\beta_1 = (\psi \circ \iota)(s_i) = \psi (\tau_i)$ for all $i \in \{1, \dots, n-1 \}$.
Since $s_1^2 = 1$ we have $\beta_1^2=1$.
Set $\beta_2 = \psi( \sigma_1)$.
From the relation $\tau_i \tau_{i+1} \sigma_i = \sigma_{i+1} \tau_i \tau_{i+1}$ it follows that $\psi (\sigma_i) = \beta_1^2\, \psi(\sigma_i) = \psi (\sigma_{i+1})\, \beta_1^2 = \psi (\sigma_{i+1})$, for all $i \in \{1, \dots, n-2 \}$, hence $\psi (\sigma_i) = \beta_2$ for all $i \in \{1, \dots, n-1 \}$.
Finally, from the relation $\tau_1 \sigma_3 = \sigma_3 \tau_1$ it follows that $\beta_1 \beta_2 = \beta_2 \beta_1$.
So, $\psi$ is Abelian.

Assume that $n = m$ and $\psi \circ \iota = \iota$, that is, $\psi (\tau_i) = \tau_i$ for all $i \in \{ 1, \dots, n-1 \}$.
For each $i \in \{1, \dots, n-1\}$ we set $\psi (\sigma_i) = \alpha_i\, \iota (w_i)$, where $\alpha_i \in \KB_n$ and $w_i \in \SSS_n$.
For each $i \in \{3, \dots, n-1\}$ we have $s_i w_1 = (\pi_K \circ \psi) (\tau_i \sigma_1) = (\pi_K \circ \psi) (\sigma_1 \tau_i) = w_1 s_i$, hence $w_1$ lies in the centralizer of $\langle s_3, \dots, s_{n-1} \rangle$ in $\SSS_n$, which is equal to $\langle s_1 \rangle = \{ 1, s_1 \}$, hence either $w_1 = s_1$ or $w_1 = 1$.

Suppose that $w_1 = s_1$, that is, $\psi (\sigma_1) = \alpha_1 \tau_1$.
For each $k \in \{3, \dots, n-1\}$ we have $\alpha_1 \tau_1 \tau_k = \psi (\sigma_1 \tau_k) = \psi (\tau_k \sigma_1) = \tau_k \alpha_1 \tau_1$, hence $s_k (\alpha_1) = \alpha_1$.
By Corollar \ref{corl3_7b} this implies that $\alpha_1 \in \KB_n [\UU_k]$, where $\UU_k =\{ \delta_{i,j} \mid i,j \not\in \{k , k+1 \} \}$.
We have $\bigcap_{3 \le k \le n-1} \UU_k = \{ \delta_{1,2}, \delta_{2,1} \}$ hence, by Theorem \ref{thm3_2}, $\alpha_1 \in \KB_n [ \{ \delta_{1,2}, \delta_{2,1} \} ]$.
Let $\omega \in F_2$ such that $\alpha_1 = \omega (\delta_{1,2}, \delta_{2,1})$.
Note that, by Theorem \ref{thm3_2}, $\KB_n [\{ \delta_{1,2}, \delta_{2,1} \}]$ is a free group freely generated by $\{ \delta_{1,2}, \delta_{2,1} \}$, hence $\omega$ is unique.
We have $\sigma_2 = \tau_1 \tau_2 \sigma_1 \tau_2 \tau_1$, hence $\psi (\sigma_2) = \omega(\delta_{2,3}, \delta_{3,2}) \tau_2$.
On the other hand, since $\sigma_1 \sigma_2 \sigma_1 = \sigma_2 \sigma_1 \sigma_2$, we have 
\begin{gather*}
\omega (\delta_{1,2}, \delta_{2,1}) \tau_1 \omega (\delta_{2,3}, \delta_{3,2}) \tau_2 \omega (\delta_{1,2}, \delta_{2,1}) \tau_1\\
 = \omega (\delta_{2,3}, \delta_{3,2}) \tau_2 \omega (\delta_{1,2}, \delta_{2,1}) \tau_1 \omega (\delta_{2,3}, \delta_{3,2}) \tau_2\,,
\end{gather*}
hence 
\begin{equation}
\omega (\delta_{1,2}, \delta_{2,1}) \omega (\delta_{1,3}, \delta_{3,1}) \omega (\delta_{2,3}, \delta_{3,2}) =
\omega (\delta_{2,3}, \delta_{3,2}) \omega (\delta_{1,3}, \delta_{3,1}) \omega (\delta_{1,2}, \delta_{2,1})\,.
\label{eq6_1}
\end{equation}
Let $\WW = \{ \delta_{i,j} \mid 1 \le i \neq j \le 3 \}$, $\WW_1 = \{ \delta_{1,2}, \delta_{2,3}, \delta_{3,1} \}$ and $\WW_2 = \{ \delta_{2,1}, \delta_{3,2}, \delta_{1,3} \}$.
By Lemma \ref{lem3_3}, $\KB_n[\WW] = \KB_n[\WW_1] * \KB_n [\WW_2]$.
From this decomposition and Lemma \ref{lem6_0} it is easily deduced that the only element of $F_2$ which satisfies Equation (\ref{eq6_1}) is $\omega = 1$, thus $\psi (\sigma_1) = \tau_1$.
Then from the equalities $\sigma_i = \tau_{i-1} \tau_i \sigma_{i-1} \tau_i \tau_{i-1}$, $2 \le i \le n-1$, it follows that $\psi (\sigma_i) = \tau_i$ for all $i \in \{2, \dots, n-1\}$, hence $\psi = \iota \circ \pi_P$.

Suppose that $w_1 =1$, that is, $\psi (\sigma_1) = \alpha_1$.
For each $k \in \{3, \dots, n-1\}$ we have $\alpha_1 \tau_k = \psi (\sigma_1 \tau_k) = \psi (\tau_k \sigma_1) = \tau_k \alpha_1$, hence $s_k (\alpha_1) = \alpha_1$.
By Corollary \ref{corl3_7b} this implies that $\alpha_1 \in \KB_n [\UU_k]$, where $\UU_k = \{\delta_{i,j} \mid i,j \not\in \{k, k+1\} \}$.
As above, $\bigcap_{3 \le k \le n-1} \UU_k = \{ \delta_{1,2}, \delta_{2,1} \}$, hence $\alpha_1 \in \KB_n [\{ \delta_{1,2}, \delta_{2,1} \}]$.
Let $\omega \in F_2$ such that $\alpha_1 = \omega (\delta_{1,2}, \delta_{2,1})$.
Again, since $\KB_n [\{ \delta_{1,2}, \delta_{2,1} \}]$ is a free group freely generated by $\{ \delta_{1,2}, \delta_{2,1} \}$, the element $\omega$ is unique. 
We have $\sigma_2 = \tau_1 \tau_2 \sigma_1 \tau_2 \tau_1$, hence $\psi(\sigma_2) = (s_1 s_2)(\alpha_1) = \omega (\delta_{2,3}, \delta_{3,2})$.
On the other hand, since $\sigma_1 \sigma_2 \sigma_1 = \sigma_2 \sigma_1 \sigma_2$, we have 
\begin{equation}
\omega (\delta_{1,2}, \delta_{2,1}) \omega (\delta_{2,3}, \delta_{3,2}) \omega (\delta_{1,2}, \delta_{2,1}) =
\omega (\delta_{2,3}, \delta_{3,2}) \omega (\delta_{1,2}, \delta_{2,1}) \omega (\delta_{2,3}, \delta_{3,2})\,.
\label{eq6_2}
\end{equation}
Recall that $\KB_n [\WW] = \KB_n [\WW_1] * \KB_n [\WW_2]$, where $\WW = \{ \delta_{i,j} \mid 1 \le i \neq j \le 3 \}$, $\WW_1 = \{ \delta_{1,2}, \delta_{2,3}, \delta_{3,1} \}$ and $\WW_2 = \{ \delta_{2,1}, \delta_{3,2}, \delta_{1,3} \}$.
From this decomposition and Lemma \ref{lem6_0} it is easily seen that, if $\omega$ satisfies Equation (\ref{eq6_2}), then $\omega$ is of the form $\omega = z^k$ with $z \in \{x,y\}$ and $k \in \Z$.
By Crisp--Paris \cite{CriPar1}, if $k \not\in \{1,0,-1\}$, then the subgroup of $\KB_n [\WW_1]$ generated by $\{ \delta_{1,2}^k, \delta_{2,3}^k, \delta_{3,1}^k \}$ is a free group freely generated by $\{ \delta_{1,2}^k, \delta_{2,3}^k, \delta_{3,1}^k \}$, hence $\delta_{1,2}^k \delta_{2,3}^k \delta_{1,2}^k \neq \delta_{2,3}^k \delta_{1,2}^k \delta_{2,3}^k$.
Similarly, we have $\delta_{2,1}^k \delta_{3,2}^k \delta_{2,1}^k \neq \delta_{3,2}^k \delta_{2,1}^k \delta_{3,2}^k$ if $k \not\in \{-1, 0, 1 \}$.
So, $\omega \in \{x, x^{-1}, y, y^{-1}, 1 \}$.
Moreover, from the equalities $\sigma_i = \tau_{i-1} \tau_i \sigma_{i-1} \tau_i \tau_{i-1}$, $2 \le i \le n-1$, it follows that $\psi (\sigma_i) = \omega (\delta_{i,i+1}, \delta_{i+1, i})$ for all $i \in \{2, \dots, n-1 \}$.
If $\omega = x$ then $\psi = \id$, if $\omega = x^{-1}$ then $\psi = \zeta_2$, if $\omega = y$ then $\psi = \zeta_1$, if $\omega=y^{-1}$ then $\psi = \zeta_1 \circ \zeta_2$, and if $\omega=1$ then $\psi = \iota \circ \pi_K$.

Now, we assume that $n = m = 6$ and $\psi \circ \iota = \iota \circ \nu_6$.
For each $i \in \{1,2,3,4,5\}$ we set $u_i = \nu_6 (s_i)$ and $\beta_i = \iota (u_i)$.
Thus, $\psi (\tau_i) = \beta_i$ for all $i \in \{1,2,3,4,5\}$.
For each $i \in \{1,2,3,4,5\}$ we set $\psi (\sigma_i) = \alpha_i \, \iota(w_i)$, where $\alpha_i \in \KB_6$ and $w_i \in \SSS_6$.
For each $i \in \{3,4,5\}$ we have $u_i w_1 = (\pi_K \circ \psi) (\tau_i \sigma_1) = (\pi_K \circ \psi) (\sigma_1 \tau_i) = w_1 u_i$, hence $w_1$ lies in the centralizer of $\langle u_3, u_4, u_5 \rangle$ in $\SSS_6$, which is equal to $\langle u_1 \rangle = \{ 1, u_1 \}$, thus either $w_1 = u_1$ or $w_1 = 1$.
On the other hand, for each $i \in \{3,4,5 \}$, we have $\alpha_1\, \iota (w_1) \, \beta_i = \psi (\sigma_1 \tau_i) = \psi (\tau_i \sigma_1) = \beta_i \alpha_1 \, \iota(w_1)$ and $\iota (w_1)\, \beta_i = \iota(w_1 u_i) = \iota(u_i w_1) = \beta_i\, \iota(w_1)$, hence $u_i (\alpha_1) = \beta_i \alpha_1 \beta_i^{-1} = \alpha_1$.
By Lemma \ref{lem6_1} it follows that $\alpha_1 = 1$.
So, either $\psi (\sigma_1) = \beta_1$ or $\psi (\sigma_1) = 1$.
If $\psi (\sigma_1) = \beta_1$, then, since $\sigma_i = \tau_{i-1} \tau_i \sigma_{i-1} \tau_i \tau_{i-1}$ for all $i \in \{2,3,4,5\}$, we have $\psi (\sigma_i) = \beta_i$ for all $i \in \{1,2,3,4,5\}$, and therefore $\psi = \iota \circ \nu_6 \circ \pi_P$.
If $\psi (\sigma_1) = 1$, then, since $\sigma_i = \tau_{i-1} \tau_i \sigma_{i-1} \tau_i \tau_{i-1}$ for all $i \in \{2,3,4,5\}$, we have $\psi (\sigma_i) = 1$ for all $i \in \{1,2,3,4,5\}$, and therefore $\psi = \iota \circ \nu_6 \circ \pi_K$.
\end{proof}



\end{document}